\documentclass[letterpaper, 10 pt, conference]{ieeeconf}  

\IEEEoverridecommandlockouts                   
\usepackage[utf8]{inputenc}
\usepackage[T1]{fontenc}
\usepackage{graphicx}

\usepackage{float}
\usepackage{amsmath}
\usepackage{amssymb}

\usepackage{mathrsfs}  
\usepackage{amsthm}
\usepackage{mathtools}
\usepackage[ruled,vlined,algo2e]{algorithm2e}

\usepackage{dsfont}
\usepackage{enumerate}
\usepackage[tight]{subfigure}
\usepackage{multirow}
\usepackage{bm}
\usepackage{color}

\usepackage{hyperref}
\usepackage{cleveref} %at the end

\newcommand{\jc}[1]{\ifthenelse{\equal{\discussion}{true}}{\textcolor{blue}{\textbf{\scriptsize [JC: #1]}}}{}}

\newcommand{\sm}[1]{\ifthenelse{\equal{\discussion}{true}}{\textcolor{green}{\textbf{\scriptsize [SM: #1]}}}{}}

\newcommand{\rj}[1]{\ifthenelse{\equal{\discussion}{true}}{\textcolor{red}{\textbf{\scriptsize [RJ: #1]}}}{}}

\newcommand{\modif}[1]{#1}

\newcommand{\vect}[1]{#1} % Vector \boldsymbol{#1}
\newcommand{\m}[1]{\mathbf{#1}} % Matrix

\newcommand{\R}{\mathbb{R}}
\newcommand{\N}{\mathbb{N}}

\newcommand{\Z}{\mathbb{Z}}

\newcommand{\myemptyset}{\varnothing}

\newtheorem{theorem}{Theorem}

\newtheorem{proposition}{Proposition}
\newtheorem{definition}{Definition}

\newtheorem{property}{Property}

\newcommand{\set}[1]{\mathcal{#1}}
\newcommand{\Sys}{\mathcal{S}}
\newcommand{\Cont}{\mathcal{C}}
\newcommand{\seq}[1]{\boldsymbol{#1}} % Vector 

\newcommand{\relOne}{alternating simulation relation} 
\newcommand{\RelOne}{Alternating simulation relation} 
\newcommand{\relOneAbr}{\operatorname{ASR}} 
\newcommand{\propOne}{controlled simulability property} 
\newcommand{\PropOne}{Controlled simulability property} 

\newcommand{\relTwo}{memoryless concretization relation} 
\newcommand{\RelTwo}{Memoryless concretization relation} 
\newcommand{\relTwoAbr}{\operatorname{MCR}} 
\newcommand{\propTwo}{memoryless concretization property} 
\newcommand{\PropTwo}{Memoryless concretization property} 
\newcommand{\controllerTwo}{memoryless concretized controller} 
\newcommand{\ControllerTwo}{Memoryless concretized controller} 

\newcommand{\relThree}{feedback refinement relation} 
 
\newcommand{\relThreeAbr}{\operatorname{FRR}}

\newcommand{\T}{\operatorname{T}}

\author{Julien Calbert \and Sébastien Mattenet \and Antoine Girard \and Raphaël M. Jungers
\thanks{*JC is a FRIA Research Fellow. This project has received funding from the H2020-EU.1.1. research and innovation programme(s) -- ERC-2016-COG under grant agreement No 725144.
RJ is a FNRS honorary Research Associate. This project has received funding from the European Research Council (ERC) under the European Union's Horizon 2020 research and innovation programme under grant agreement No 864017 - L2C.
}
\thanks{
J.~Calbert, S.~Mattenet and R.~M.~Jungers are with the ICTEAM Institute, UCLouvain. A. Girard is with the L2S-CNRS, CentraleSupélec, Université Paris-Saclay.
{\tt\small \{julien.calbert,sebastien.mattenet,}
{\tt\small raphael.jungers\}@uclouvain.be},
{\tt\small antoine.girard@centralesupelec.fr}}
} 

\newcommand{\discussion}{true}

\begin{document}

\title{\RelTwo}

\maketitle

\begin{abstract}
% We introduce the notion of \relTwo{}, generalizing the \relThree{}. 

% Structure:
% 1) WFRR est un cas particulier de ASR
% qui permet une concretization efficace.
% 3) La diff avec ASR n'a de sens que si une cover abstraction.
% 4) On put etendre une ASR en WFRR
% 2) WFRR est une généralisation de FRR.
% 5) On a un if et only if avec la propriété

We introduce the concept of \relTwo{} ($\relTwoAbr$) to describe abstraction within the context of controller synthesis. This relation is a specific instance of~\relOne{} ($\relOneAbr$), where it is possible to simplify the controller architecture. In the case of $\relOneAbr$, the concretized controller needs to simulate the concurrent evolution of two systems, the original and abstract systems, while for~$\relTwoAbr$, the designed controllers only need knowledge of the current concrete state.
We demonstrate that the distinction between $\relOneAbr$ and $\relTwoAbr$ becomes significant only when a non-deterministic quantizer is involved, such as in cases where the state space discretization consists of overlapping cells. We also show that any abstraction of a system that alternatingly simulates a system can be completed to satisfy~$\relTwoAbr$ at the expense of increasing the non-determinism in the abstraction.
We clarify the difference between the $\relTwoAbr$ and the \relThree{} ($\relThreeAbr$), showing in particular that the former allows for non-constant controllers within cells.
% This concept extends the idea of \relThree{} since the controllers designed do not solely rely on quantified (or symbolic) state information, allowing for concrete, piecewise state-dependent controllers. 
This provides greater flexibility in constructing a practical abstraction, for instance, by reducing non-determinism in the abstraction.
Finally, we prove that this relation is not only sufficient, but also necessary, for ensuring the above properties. 

% for ensuring the functionality of this controller architecture but also necessary if it is to work with all specifications.

\end{abstract}
% % SECTION 1 : INTRODUCTION ----------------------------------
\section{Introduction}\label{Sec:Intro}

\emph{Abstraction-based control} techniques involve synthesizing a correct-by-construction controller through a systematic three-step procedure illustrated on~\Cref{fig:abstraction-procedure}. %
First, both the original system and the specifications are transposed into an abstract domain, resulting in an \emph{abstract system} and corresponding abstract specifications.
In this paper, we refer to the original system as the \emph{concrete} system as opposed to the abstract system.
Next, an abstract controller is synthesized to solve this abstract control problem. Finally, in the third step, called \emph{concretization}\footnote{Note that in other works such as~\cite{reissig2016feedback}, this step is called \emph{refinement procedure}.} as opposed to \emph{abstraction}, a controller for the original control problem is derived from the abstract controller.
% advantage 
The value of this approach lies in the substitution of the concrete system (often a system with an infinite number of states) with a finite state system, which makes it possible to leverage powerful control tools (see~\cite{belta2017formal,kupferman2001model}), for example from graph theory, such as Dijkstra, the $A^*$ algorithm or dynamic programming, allowing to design correct-by-construction controllers, often with rigorous (safety, or performance) guarantees.
Abstraction-based controller design crucially relies on the existence of a \emph{simulation} relation between the concrete and abstract states, allowing to infer the control actions in the concrete system from the ones actuated in the corresponding states of the abstract system.
Most approaches are based on the \emph{\relOne{}} ($\relOneAbr$)~\cite{alur1998alternating}, which provides a framework for guaranteeing a concretization step that provably ensures the specifications for the concrete system. 

\begin{figure}
    \centering
    \includegraphics[width=0.48\textwidth]{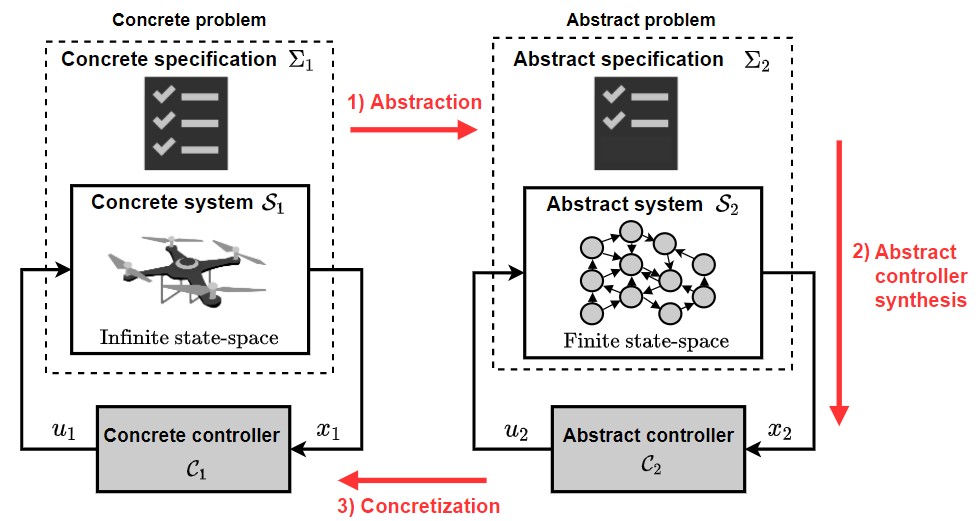}
    \caption{The three steps of abstraction-based control. Our work focuses on the algorithmic burden necessary for the third step, and characterizes the relation between the systems allowing for a memoryless algorithm at step three.}
    \label{fig:abstraction-procedure}%
\end{figure}

%%%%%%%%%%%%%%%%%%%%%%%%%%%%%%%%%%%%%%%%%
%%%%%%%%% Limitations de ASR %%%%%%%%%%%%
%%%%%%%%%%%%%%%%%%%%%%%%%%%%%%%%%%%%%%%%%

Although the \relOne{} offers a safety critical framework, it has several practical drawbacks. The first issue is that this relation provides no complexity guarantee on the concretization procedure, i.e., the concrete controller could contain the entire abstraction (which is typically made up of millions of states and transitions) as a building block, we refer to this problem as the \emph{concretization complexity issue}. 
See \cite[Proposition 8.7]{tabuada2009verification} for a description of the algorithmic concretization procedure.
The second limitation is that, although $\relOneAbr$ guarantees the existence of a controller for the concrete system such that the closed-loop systems share the same behavior, it does not guarantee that some of these properties will be transferred to the concrete controller during the concretization step.

For this reason, various relations have been introduced in the literature, e.g. \cite{rungger2016scots, borri2018design, egidio2022state, majumdar2020abstraction}, to tackle specific shortcomings of the \relOne{}. %
Some works propose abstraction-based techniques that do not suffer from this concretization complexity issue, but they are either limited to subsets of specification such as safety or reachability~\cite{girard2012controller, dallal2013supervisory, grune2007approximately,hsu2018lazy}, or for a specific class of dynamical systems, e.g. incrementally stable ~\cite{girard2012controller}, piecewise
affine (PWA) dynamics~\cite{yordanov2011temporal}.
Other abstraction-based controller synthesis methods circumvent the concretization complexity issue by constructing abstractions without overlapping cells~\cite{girard2013low, yordanov2011temporal, grune2007approximately, hsu2019lazy, legat2021abstraction, calbert2021alternating}.
In particular, in~\cite{rungger2016scots}, the authors present the~\emph{\relThree{}} ($\relThreeAbr$) which provides a framework for an extremely simple controller architecture, which is not restricted to certain types of specifications or systems, and, as we shall see, forms the basis of our reflection.

% For this reason, various relations have been introduced in the literature, e.g. \cite{rungger2016scots, egidio2022state, majumdar2020abstraction}, to tackle specific shortcomings of the \relOne{}. %
% %
% Nevertheless, these relations are often introduced in the context of a specific abstraction algorithm rather than in a generic framework, and it is not clear how they provide answers to the two shortcomings described above. 

% What we do
In this paper, we study the properties that a simulation relation might require, and establish their implications on controller design characteristics. Specifically, we analyze how the relation established between the concrete and abstract systems during the abstraction phase affects the concretization process, while the methods used to compute abstractions or solve the resulting abstract control problems are beyond the scope of this work.
We introduce the~\emph{\relTwo{}}~($\relTwoAbr$) which guarantees a simple control architecture for the concrete system that is not problem-specific. Unlike the general case of~\relOne{}, the controller only needs information about the current concrete state and does not need to simulate the concurrent evolution of the concrete and abstract systems. 
We show that in the specific case of a deterministic quantizer (that is, a single-valued map), \relOne{} and \relTwo{} coincide. As a consequence, the~\relTwo{} turns out to be essential in the context of an abstraction composed of overlapping cells, a framework we consider crucial for building a non-trivial smart abstraction, see~\cite{egidio2022state}. 
We prove that $\relTwoAbr$ is not only sufficient to guarantee the functionality of this memoryless controller architecture, but also necessary for it to work with all abstract controllers, i.e.,  regardless of the specifications.
Finally, we propose an in-depth discussion of~\relThree{} which turns out to be a special case of \relTwo{}, with the additional requirement to use only symbolic state information.  As a result, \relThree{} is limited to the design of piecewise constant controllers, whereas~\relTwo{} allows the design of piecewise state-dependent controllers. We illustrate on a simple example the advantage of~\relTwo{} over \relThree{} when concrete state information is available.

\noindent \textbf{Notation:} 
% Intervals
The sets $\R,\Z, \Z_+$ denote respectively the sets of real numbers, integers and non-negative integer numbers.
For example, we use $[a,b]\subseteq \R$ to denote a closed continuous interval and $[a;b] = [a,b]\cap \Z$ for discrete intervals. The symbol $\myemptyset$ denotes the empty set.
% Set-valued map
Given two sets $A,B$, we define a \emph{single-valued map} as $f:A\rightarrow B$, while a~\emph{set-valued map} is defined as $f:A\rightarrow 2^B$, where $2^B$ is the power set of $B$, i.e., the set of all subsets of $B$. 
The image of a subset $\Omega\subseteq A$ under $f:A\rightarrow 2^B$ is denoted $f(\Omega)$.
% Relation
We identify a binary \emph{relation} $R\subseteq A \times B$ with set-valued maps, i.e., $R(a) = \{b \mid (a, b)\in R\}$ and $R^{-1}(b) =\{ a \mid (a, b)\in R\}$.
Given a set $A$, we denote the binary \emph{identity} relation $Id_A\subseteq A\times A$ such that $(x , y)\in Id_A \Leftrightarrow x=y$.
A relation $R\subseteq A\times B$ is \emph{strict} and \emph{single-valued} if for every $a\in A$ the set $R(a)\ne \myemptyset$ and $R(a)$ is a singleton, respectively. 
% Operators
Given two set-valued maps~$f,g$, we denote by $f\circ g$ their composition $(f\circ g)(x) = f(g(x))$.
If $R\subseteq A\times B$ and $Q\subseteq B\times C$, then their composition $R\circ Q =\{(a, c)\in A\times C\ \mid \exists b\in B \text{ such that } (a,b)\in R \text{ and } (b, c)\in Q\}$.

% % SECTION 2 : PRELIMINARIES ---------------------------------
\section{Preliminaries}\label{Sec:Preliminairies}

In this section, we start by defining the considered control framework, i.e., the systems, the controllers and the specifications.

We consider dynamical systems of the following form.
\begin{definition}\label{def:sys}
A \emph{transition control system} is a tuple $\Sys\coloneqq (\set{X}, \set{U}, F)$ where $\set{X}$ and $\set{U}$ are respectively the set of states and inputs and the set-valued map $F:\set{X}\times \set{U}\rightarrow 2^{\set{X}}$ where $F(x,u)$ give the set of states that can be reached from a given state $\vect x$ under a given input $\vect u$.
\hfill$\bigtriangledown$
\end{definition}
We introduce the set-valued operator of  \emph{available inputs}, defined as $\mathcal{U}_F(x) = \{ \vect u \in \set{U} \mid F(x,u) \neq \myemptyset \}$, which give the set of inputs $\vect u$ available at a given state $\vect x$. When it is clear from the context which system it refers to, we simply write the available inputs operator $\set{U}(x)$.
In this paper, we consider \emph{non-blocking} systems, i.e., $\forall x\in \set{X}: \set{U}(x)\ne \myemptyset$.

The use of a set-valued map to describe the transition map of a system allows to model perturbations and any kind of non-determinism in a common formalism.
% Properties of specific systems
We say that a transition control system is \emph{deterministic} if for every state $\vect x \in \set{X}$ and control input $\vect u \in \set{U}$, $F(x,u)$ is either empty or a singleton. Otherwise, we say that it is \emph{non-deterministic}. %
A \emph{finite-state} system, in contrast to an \emph{infinite-state} system, refers to a system characterized by finitely many states and inputs.

% Definition of a solution
A tuple $(\seq x, \seq u)\in \set{X}^{[0; T[}\times \set{U}^{[0; T-1[}$ is a \emph{trajectory} of length $T$ of the system $\Sys=(\set{X},\set{U}, F)$ starting at $x(0)$ if $T\in\N\cup \{\infty\}$, $x(0)\in\set{X}$, $\forall k\in [0;T-1[: u(k)\in \set{U}(x(k))$ and $x(k+1) \in F(x(k), u(k))$.
The set of trajectories of $\Sys$ is called the \emph{behavior} of $\Sys$, denoted $\set{B}(\Sys)$.

% Definition of the behavior of a system
\begin{definition}\label{def:behavior}
Let $\Sys=(\set{X}, \set{U}, F)$. The set $\set{B}(\Sys)=\{ (\seq x,\seq u) \mid \exists T \in \N\cup\{\infty\}\text{ such that } (\seq x,\seq u)\text{ is a trajectory of }\Sys \text{ of length } T\},$
is called the \emph{behavior} of $\Sys$.
We define the \emph{state behavior} $\set{B}_x(\Sys)$ as $\seq x\in \set{B}_x(\Sys)\Leftrightarrow \exists \seq u: (\seq x,\seq u)\in \set{B}(\Sys)$.
\end{definition}

% No outputs
It is a common practice in the abstraction-based framework to represent systems with an extra set $\set{Y}$ for outputs, and an output mapping function $H:\set{X}\rightarrow \set{Y}$. In that case, the states in $\set{X}$ are seen as internal to the system, while the outputs are externally visible. However, as this is not essential to our current discussion, we exclude it here.

% No internal variables
We consider systems without \emph{internal variables}, see \cite[Definition~III.1]{reissig2016feedback} and the discussion that follows this definition. This extension is necessary to correctly define the composition of a system with a dynamic controller. For the sake of clarity and to avoid the use of internal variables, we only consider static controllers.
%~\footnote{The definitions and proofs easily extend to the case of internal variables.}.

% Controller's description
\begin{definition}
 We define a \emph{static controller} for a system $\Sys = (\set{X},\set{U}, F)$ as a set-valued map $\Cont :\set{X}\rightarrow 2^\set{U}$ such that $\forall \vect x\in\set{X}: \ \Cont(\vect x) \subseteq \set{U}(x)$, $\Cont(x) \ne \myemptyset$. We define the \emph{controlled system}, denoted as $\Cont\times \Sys$, as the transition system characterized by the tuple $(\set{X},\set{U}, F_{\Cont})$ where
 $x'\in F_{\Cont}(x, u) \Leftrightarrow (u\in \Cont(x)\land x'\in F(x,u))$.
\end{definition}
This controller is \emph{static} since the set of control inputs enabled at a given state only depend on the state. A more general notion of controller (called a \emph{dynamic} controller)
would take the entire past trajectory and/or other variables as input. For the sake of clarity we limit ourselves to static controllers although most definitions and proofs extend easily to dynamical controllers.

We now define the control problem.
\begin{definition}\label{def:specification}
Consider a system $\Sys = (\set{X},\set{U}, F)$. A \emph{specification} $\Sigma$ for $\Sys$ is defined as any subset $\Sigma \subseteq (\set{X}\times \set{U})^{\infty}$. It is said that system $\Sys$ \emph{satisfies} the specification $\Sigma$ if $\set{B}(\Sys)\subseteq \Sigma$. 
A system $\Sys$ together with a specification $\Sigma$ constitute a \emph{control problem} $(\Sys,\Sigma)$.
Additionally, a controller $\Cont$ is said to \emph{solve} the control problem $(\Sys,\Sigma)$ if $\Cont\times \Sys$ satisfies the specification $\Sigma$.
\end{definition}

Given $\set{X}_I, \set{X}_T, \set{X}_O \subseteq \set{X}$, we define the specific \emph{reach-avoid} specification that we will use to motivate new relations
\begin{align}
\Sigma^{\text{Reach}} = \{ &(\seq x,\seq u) \in (\set{X} \times \set{U})^\infty \mid x(0) \in \set{X}_I \Rightarrow \nonumber \\
&\exists k \in \Z_+ : \left(x(k) \in \set{X}_T \land \forall k' \in [0;k):x(k')\notin \set{X}_O\right) \}, \label{eq:reach-avoid-spec}
\end{align}
which enforces that all states in the initial set $\set{X}_I$ will reach the target $\set{X}_T$ in finite time while avoiding obstacles in $\set{X}_O$. We use the abbreviated notation $\Sigma^{\text{Reach}}=[\set{X}_I,\set{X}_T,\set{X}_O]$ to denote the specification~\eqref{eq:reach-avoid-spec}.  

% % SECTION 3 : PRELIMINARIES ---------------------------------
\section{\RelOne}\label{Sec:ASR}

%%%%%%%%%%%%%%%%%%%%%%%%%%%%%%%%%%%%%%%%%%
% Introduce abstraction-based control
%%%%%%%%%%%%%%%%%%%%%%%%%%%%%%%%%%%%%%%%%%
In this paper, we will refer to $\Sys_1=(\set{X}_1, \set{U}_1,F_1)$ as the concrete system and $\Sys_2=(\set{X}_2, \set{U}_2, F_2)$ as its abstraction.

In practice, the abstract domain $\set{X}_2$ of $\Sys_2$ is constructed by discretizing the concrete state space $\set{X}_1$ of $\Sys_1$ into subsets (called \emph{cells}). The discretization is induced by the relation $R\subseteq \set{X}_1\times \set{X}_2$, i.e., the cell associated with the abstract state $x_2\in\set{X}_2$ is $R^{-1}(x_2)\subseteq \set{X}_1$.
Note that in this context, we refer to the set-valued map $R(x_1) = \{x_2 \mid (x_1, x_2)\in R\}$ as the \emph{quantizer}.
When $R$ is a single-valued map, we refer to it as defining a \emph{partition} of $\set{X}_1$, in contrast to the case of set-valued maps where we say that it defines a \emph{cover} of $\set{X}_1$, see~\Cref{fig:discretization} for clarity.
Notably, the condition that $R$ is a strict relation is equivalent to ensuring that the discretization completely covers $\set{X}_1$.

\begin{figure}
    \centering
    \includegraphics[width=0.42\textwidth]{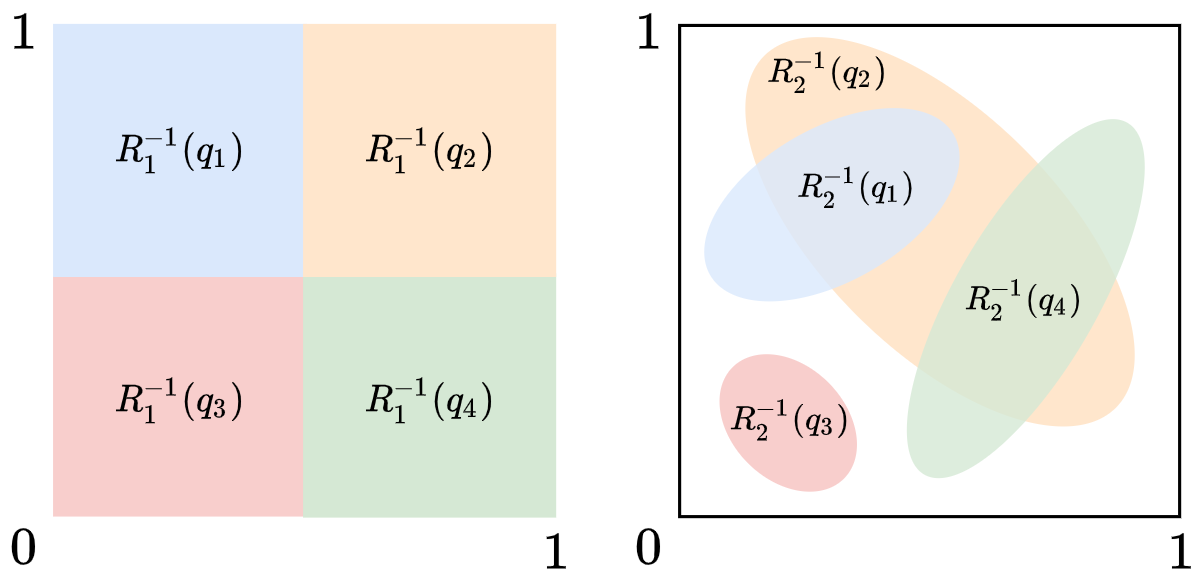}
    \caption{Types of discretization of the concrete state space.
    Let $\Sys_{1} = (\set{X}_1,\set{U}_1,F_1)$ with $\set{X}_1 = [0,1]^2$, $\Sys_{2} = (\set{X}_2,\set{U}_2,F_2)$ with $\set{X}_2 = \{q_1,q_2,q_3,q_4\}$, $R_1 \subseteq \set{X}_1\times \set{X}_2$ and $R_2 \subseteq \set{X}_1\times \set{X}_2$ are explicit from the figure.
    Left: $R_1$ is a strict single-valued map, i.e., it induces a full partition of $\set{X}_1$. Right: $R_2$ is a non-strict set-valued map, i.e., it induces a partial cover of $\set{X}_1$.
    }
    \label{fig:discretization}
\end{figure}

% Specification
As mentioned in the introduction, the concrete specification $\Sigma_1$ must also be translated into an abstract specification $\Sigma_2$.
For example, given $\Sigma_1^{\text{Reach}} = [\set{X}_{I},\set{X}_{T},\set{X}_{O}]$ with $\set{X}_{I},\set{X}_{T}, \set{X}_{O}\subseteq \set{X}_1$, the abstract specification $\Sigma_2^{\text{Reach}} = [\set{Q}_{I},\set{Q}_{T},\set{Q}_{O}]$ with $\set{Q}_{I},\set{Q}_{T}, \set{Q}_{O}\subseteq \set{X}_2$ must satisfy the following conditions
$R(\set{X}_{I})\subseteq \set{Q}_{I}$, $\set{Q}_{T}\subseteq R(\set{X}_{T})$ and $R(\set{X}_{O})\subseteq \set{Q}_{O}$.

% Control
Given an abstract controller $\Cont_2$ that solves the control problem $(\Sys_2,\Sigma_2)$, to guarantee the existence of a controller $\Cont_1$ that solves the concrete problem $(\Sys_1,\Sigma_1)$, the tuple 
% $(\Sys_1,\Sys_2,\Cont_1,\Cont_2, R)$
\modif{$(\Sys_1,\Sys_2, R)$} 
must satisfy the following \emph{\propOne}, which guarantees that the behavior of the concrete controlled system $\Cont_1 \times \Sys_1$ is simulated by the abstract controlled system $\Cont_2 \times \Sys_2$ via the relation~$R$. 

%%%%%%%%%%%%%%%%%%%%%%%%%%%%%%%%%%%%%%%%%%%%%%%%%%%
%%%%%%%%%%%%%%%%%%%% property %%%%%%%%%%%%%%%%%%%%%
%%%%%%%%%%%%%%%%%%%%%%%%%%%%%%%%%%%%%%%%%%%%%%%%%%%

\begin{property}[\PropOne-{\cite[(5)]{reissig2016feedback}}]\label{prop:reproductibility}
\modif{Given two systems $\Sys_1$ and $\Sys_2$, 
%their respective controllers $\Cont_1$ and $\Cont_2$ 
and a relation $R\subseteq \set{X}_1\times \set{X}_2$, we say that the tuple 
% $(\Sys_1, \Sys_2, \Cont_1, \Cont_2, R)$ 
$(\Sys_1, \Sys_2, R)$ 
satisfies the \emph{\propOne} if for every controller $\Cont_2$ there exists a (possibly non-static) controller $\Cont_1$ such that 
for every trajectory $\seq x_1$ of length $T$ of the controlled system $\Cont_1\times \Sys_1$, there exists a trajectory $\seq x_2$ of length $T$ of the controlled system $\Cont_2\times \Sys_2$ satisfying 
\begin{equation}\label{eq:controlled_simulability}
\forall k\in[0;T-1]:\ (\vect x_1(k),\vect x_2(k))\in R.
\end{equation}
The condition~\eqref{eq:controlled_simulability} can be equivalently reformulated as follows
% Equivalently, the tuple 
% %$(\Sys_1, \Sys_2, \Cont_1, \Cont_2, R)$ 
% $(\Sys_1, \Sys_2, R)$ 
% satisfies the \emph{\propOne} if 
\begin{equation}\label{eq:propOneSet}
\set{B}_x(\Cont_1\times \Sys_1)\subseteq R^{-1}(\set{B}_x(\Cont_2\times \Sys_2)).
\end{equation}}
\end{property}

Controlled simulability can be used to guarantee that properties satisfied by the abstract controlled system $\Cont_2\times \Sys_2$ are also satisfied by the concrete controlled system $\Cont_1\times \Sys_1$.
To design $\Cont_1$ from $\Cont_2$ in~\Cref{prop:reproductibility}, 
% To design $\Cont_1$ from $\Cont_2$ 
%so that $(\Sys_1,\Sys_2,\Cont_1,\Cont_2,R)$ satisfies the \propOne, 
the relation must impose conditions on the local dynamics of the systems in the associated states, accounting for the impact of different input choices on state transitions. The \emph{\relOne}~\cite[Definition 4.19]{tabuada2009verification} is a comprehensive definition of such a relation. It guarantees that any control applied to $\Sys_2$ can be concretized into a controller for $\Sys_1$ that maintains the relation between the controlled trajectories.

\begin{figure}
   \centering
   \includegraphics[width=0.48\textwidth]{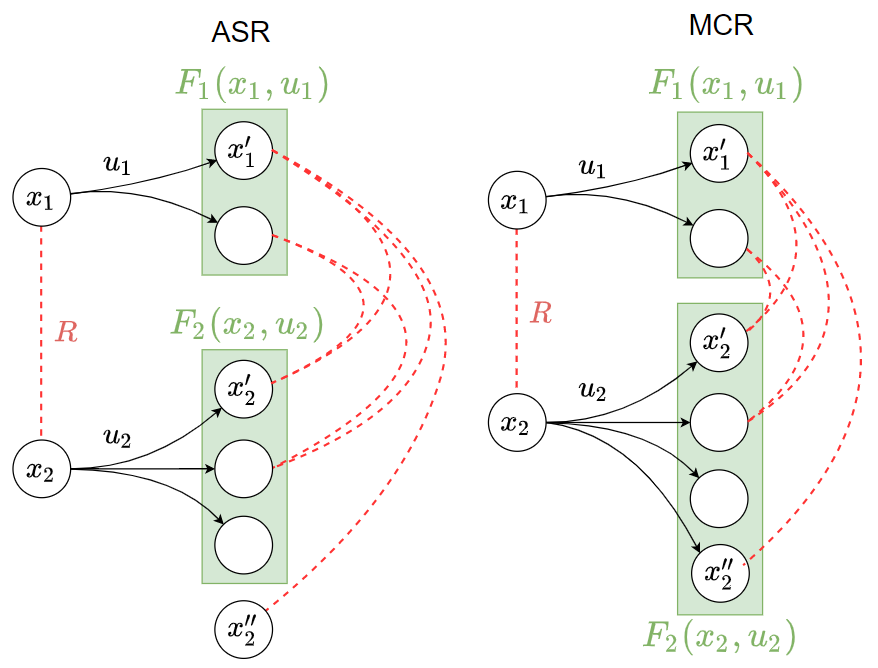}
   \label{fig:blabla}
  \caption{Local conditions between related states to satisfy the relation. Note that here, the relation is fixed (i.e., here a cover of the state space since $|R(x_1')|>1$) as well as the transition map of the concrete system and that we are simply changing the transition map of the abstract system. 
  Left: The relation $R$ is a $\relOneAbr$  but not a $\relTwoAbr$. 
  Right: The relation $R$ is a~$\relTwoAbr$.
  }
  \label{fig:local-condition}
\end{figure}

%%%%%%%%%%%%%%%%%%%%%%%%%%%%%%%%%%%%%%%%%%%%%%%%%%%
%%%%%%%%%%%%%%%%%%%% ASR def %%%%%%%%%%%%%%%%%%%%%%
%%%%%%%%%%%%%%%%%%%%%%%%%%%%%%%%%%%%%%%%%%%%%%%%%%%

\begin{definition}[$\relOneAbr$]\label{def:R1:ASR}
 A relation $R \subseteq \set{X}_1\times \set{X}_2$ is an \emph{\relOne} from $\Sys_1$ to $\Sys_2$ if 
 for every $(\vect x_{1},\vect x_{2}) \in R$
       and for every $\vect u_2\in \set{U}_2(x_2)$ there exists $\vect u_1\in \set{U}_1(x_1)$ such that for every $\vect x_1'\in F_1(x_1,u_1)$ there exists $\vect x_2'\in F_2(x_2,u_2)$ such that $(\vect x_1',\vect x_2')\in R$.
\end{definition}
Note that sometimes one prefers to speak of the \emph{extended relation} $R_e\subseteq \set{X}_1\times \set{X}_2\times \set{U}_1\times \set{U}_2$, which is defined by the set of $(\vect x_1, \vect x_2, \vect u_1, \vect u_2)$ satisfying the condition given in the definition of~$R$. When it refers to a $\relOneAbr$, we denote it~$R_e^{\relOneAbr}$.

The local condition stated in~\Cref{def:R1:ASR} is illustrated for some $(x_1,x_2,u_1,u_2)\in R_e^{\relOneAbr}$ in Figure~\ref{fig:local-condition} (left).
The fact that $R$ is an \relOne{} from $\Sys_1$ to $\Sys_2$ will be denoted $\Sys_1 \preceq_R^{\relOneAbr} \Sys_2$, and we write $\Sys_1 \preceq^{\relOneAbr} \Sys_2$ if $\Sys_1 \preceq_R^{\relOneAbr} \Sys_2$ holds for some $R$.

Note that this definition slightly differs from~\cite[Definition 4.19]{tabuada2009verification}, where there is an additional requirement regarding the outputs of related states. This requirement is useful for the concept of approximate \relOne{} \cite[Definition 9.6]{tabuada2009verification} which we do not discuss here.

One needs an \emph{interface} to map abstract inputs to concrete ones.
\begin{definition}[Interface]
Given two systems $\Sys_1$ and $\Sys_2$, a relation $R$ \modif{of type $\T$, i.e. $\Sys_1\preceq_R^{\T}\Sys_2$}, and its associated extended relation $\modif{R_e^{\T}}$, % $\Sys_1\preceq_R\Sys_2$, 
a map $\modif{I_R^{\T}}:\set{X}_1\times \set{X}_2\times \set{U}_2 \rightarrow 2^{\set{U}_1}$ is an \emph{interface} from $\Sys_2$ to $\Sys_1$ if:\\
$\forall (x_1,x_2)\in R, \ \forall u_2\in\set{U}_2(x_2)$, 
\begin{align}
\modif{I_R^{\T}}(x_1, x_2, u_2)&\ne \myemptyset \label{eq:interface:non-empty}\\
\modif{I_R^{\T}}(x_1, x_2, u_2) &\subseteq \{u_1 \mid (x_1,x_2,u_1,u_2)\in \modif{R_e^{\T}}\}. \label{eq:interface:subset}
\end{align}
% where $R_e$ is the extended relation associated to $R$.
\end{definition}
The \emph{maximal interface} associated to $R$ satisfies $\modif{I_R^{\T}}(x_1, x_2, u_2) = \{u_1 \mid (x_1,x_2,u_1,u_2)\in \modif{R_e^{\T}}\}$.

\begin{theorem}\label{th:ASR-reproductibility}
Let two systems $\Sys_1$ and $\Sys_2$, and a relation $R$ such that $\Sys_1\preceq_R^{\relOneAbr} \Sys_2$, an interface $\modif{I_R^{\relOneAbr}}$, and any trajectories $(\seq x_1, \seq u_1)$ and $(\seq x_2, \seq u_2)$ of length $T$ where $T\in\N\cup \{\infty\}$ such that $(x_1(0), x_2(0))\in R$ and
    \begin{align*}
        u_2(k)&\in \set{U}_2(x_2(k)) & k\in [0;T-1[;\\
        u_1(k)&\in \modif{I_R^{\relOneAbr}}(x_1(k), x_2(k), u_2(k))& k\in [0;T-1[;\\
        x_1(k+1)&\in F_1(x_1(k), u_1(k))& k\in [0;T-1[;\\
        x_2(k+1)&\in F_2(x_2(k), u_2(k)) \cap R(x_1(k+1))& k\in [0;T-1[.
    \end{align*}
Then, the following holds
\begin{align*}
    (\seq x_1, \seq u_1)& \in \set{B}(\Sys_1),\\
    (\seq x_2, \seq u_2)& \in \set{B}(\Sys_2),\\
    \forall k\in [0;T-1]:& \ (x_1(k), x_2(k))\in R.
\end{align*}
\end{theorem}
\begin{proof}
We proceed by induction on $k$.
Given that $(x_1(k),x_2(k))\in R$, which is true for $k=0$.
Since $\Sys_1\preceq_R^{\relOneAbr}\Sys_2$, and according to~\eqref{eq:interface:non-empty} and~\eqref{eq:interface:subset}, for every $u_2(k)\in \set{U}_2(k)$, there exists $u_1(k)\in \modif{I_R^{\relOneAbr}}(x_1(k),x_2(k),u_2(k))\subseteq \set{U}_1(x_1(k))$.
Furthermore, since $\Sys_1\preceq_R^{\relOneAbr}\Sys_2$, we can conclude that $F_2(x_2(k), u_2(k)) \cap R(x_1(k+1))\ne \myemptyset$, ensuring that $(x_1(k+1), x_2(k+1))\in R$.
\end{proof}

\begin{theorem}\label{th:ASR:sufficient}
    \modif{Given two systems $\Sys_1$ and $\Sys_2$, if $\Sys_1 \preceq_R^{\relOneAbr} \Sys_2$, 
    %then for every controller $\Cont_2$ there exists a (possibly non-static) controller $\Cont_1$ such that $(\Sys_1, \Sys_2, \Cont_1, \Cont_2, R)$ satisfies the \propOne{} (\Cref{prop:reproductibility}).
    then $(\Sys_1, \Sys_2, R)$ satisfies the \propOne{} (\Cref{prop:reproductibility}).}
\end{theorem}
\begin{proof}
    \modif{
    Given a controller $\Cont_2$, we can define a controller $\Cont_1$ implementing the algorithm given in~\Cref{th:ASR-reproductibility} by taking $u_2(k)\in \Cont_2(x_2(k))\subseteq \set{U}_2(x_2(k))$, which ensure that $(\Sys_1, \Sys_2, R)$ satisfies the \propOne{}. 
    }
    % We can define a controller $\Cont_1$ implementing the algorithm given in~\Cref{th:ASR-reproductibility} by taking $u_2(k)\in \Cont_2(x_2(k))\subseteq \set{U}_2(x_2(k))$, which ensure that $(\Sys_1, \Sys_2, \Cont_1, \Cont_2, R)$ satisfies the \propOne{}. 
\end{proof}

\begin{figure}
    \centering
    \includegraphics[width=0.49\textwidth]{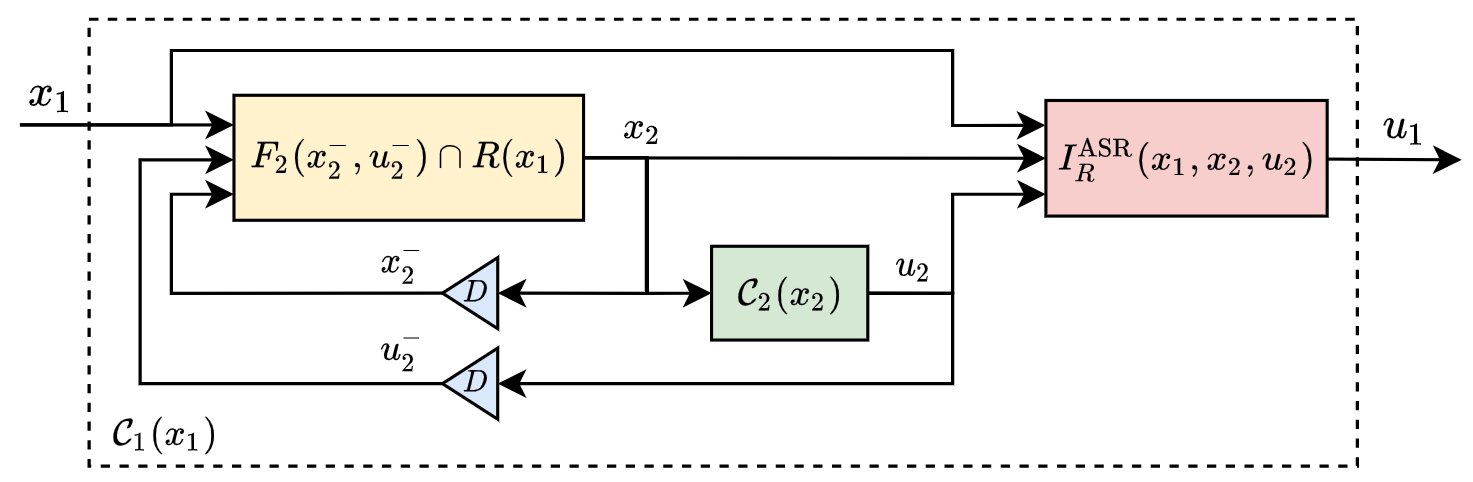}
    \caption{Dynamic controller architecture $\Cont_1$ such that $(\Sys_1,\Sys_2,\Cont_1,\Cont_2,R)$ satisfies the \propOne{} if $\Sys_1\preceq_R^{\relOneAbr} \Sys_2$.
    The red block is the \emph{interface} $\modif{I_R^{\relOneAbr}}$.
    The yellow block contains the \emph{abstract system} and the \emph{quantizer}.
    The green block is the \emph{abstract controller}.
    The blue blocks are \emph{delay blocks} which represent memory elements that store and provide a past value of a signal, e.g., $D(x_2(k)) = x_2(k-1)$.}
    \label{fig:controller-architecture-ASR}
\end{figure}

From~\Cref{th:ASR-reproductibility}, we can construct a concrete controller $\Cont_1$, whose block diagram is given in~\Cref{fig:controller-architecture-ASR}, guaranteeing the \propOne{}. 
This concrete controller simulates the abstraction at every time step which can be computationally expensive if $F_2$ is hard to compute, e.g., if $F_1$ itself is hard to compute (or approximate).
This is what we refer to as the \emph{concretization complexity issue}.  
% Memory
In addition, the proposed controller architecture $\Cont_1$ requires memory to account for last past abstract input and state, as illustrated on the block diagram in~\Cref{fig:controller-architecture-ASR}.

%%%%%%%%%%%%%%%%%%%%%%%%%%%%%%%%%%%%%%%%%%%%%%%%%%%
%%%%%%%%%%%%%%%%%%%% refinement %%%%%%%%%%%%%%%%%%%
%%%%%%%%%%%%%%%%%%%%%%%%%%%%%%%%%%%%%%%%%%%%%%%%%%%

Note that $\Sys_1 \preceq^{\relOneAbr} \Sys_2$ guarantees that \modif{$(\Sys_1,\Sys_2,R)$ satisfies \propOne{}.}
% for any controller $\Cont_2$ there exists a controller $\Cont_1$ satisfying the \propOne{}.
Nevertheless, this does not guarantee that any arbitrary property of $\Cont_2$ will be translated in $\Cont_1$. In particular the property of a controller being static is certainly not preserved. This will be illustrated in the next section. Another downside is that the controller given here requires simulating the abstract system at every step to know which inputs to allow in the concrete case.

\modif{
\Cref{th:ASR-reproductibility} does not appear formally in~\cite{tabuada2009verification},
but it is suggested in the discussion that follows~
\cite[Def 6.1-Feedback composition]{tabuada2009verification}.
The content of~\Cref{th:ASR:sufficient} was mostly implied in the discussion around
\cite[Prop 8.7]{tabuada2009verification}.
In this light, \Cref{Sec:ASR} gives the necessary background, in standardized form, for the introduction of the~\relTwo{}.
}

% % SECTION 4 : PRELIMINARIES ---------------------------------
\section{\RelTwo}\label{Sec:WFRR}
We start by defining the concrete controller resulting from a specific concretization scheme.
\begin{definition}[\ControllerTwo]
Given two systems $\Sys_1$ and $\Sys_2$, a strict relation $R$ \modif{of type $\T$}, an interface $\modif{I_R^{\T}}$ and a controller $\Cont_2$, we define the \emph{\controllerTwo{}}, denoted $\Cont_1 = \Cont_2\circ_{\modif{I_R^{\T}}} R $ as the mapping 
\begin{equation}\label{eq:local-refined-controller}
    \Cont_1(x_1) = (\Cont_2\circ_{\modif{I_R^{\T}}} R)(x_1) =\bigcup_{x_2\in R(x_1)} \modif{I_R^{\T}}(x_1,x_2,\Cont_2(x_2)).
\end{equation}
%where $I_R$ is any interface associated to $R$.
\end{definition}
The term \emph{memoryless} assigned to $\Cont_1$ is justified by the observation that $\Cont_1$ exclusively makes decisions based on the current concrete state. 

% \ControllerTwo{} can be interpreted as a \emph{hybrid} controller because it requires both concrete state (usually from an infinite-state space) and abstract state (from a finite-state space) measurements to operate.
%and it generates continuous-time and discrete-time inputs to control

\subsection{Motivation}

The example below illustrates that the \relOne{} guarantees the ability to derive a controller $\Cont_1$ for $\Sys_1$ \modif{from any controller $\Cont_2$ of $\Sys_2$, i.e., $(\Sys_1,\Sys_2, R)$ satisfies the \propOne{}.} However, it also points out that the specific associated \controllerTwo{} $\modif{\Cont_1 = \Cont_2\circ_{I_R^{\relOneAbr}} R}$ does not \modif{guarantee~\eqref{eq:propOneSet}}.
%the \propOne{}. 

\begin{figure}
    \centering
    \includegraphics[width=0.49\textwidth]{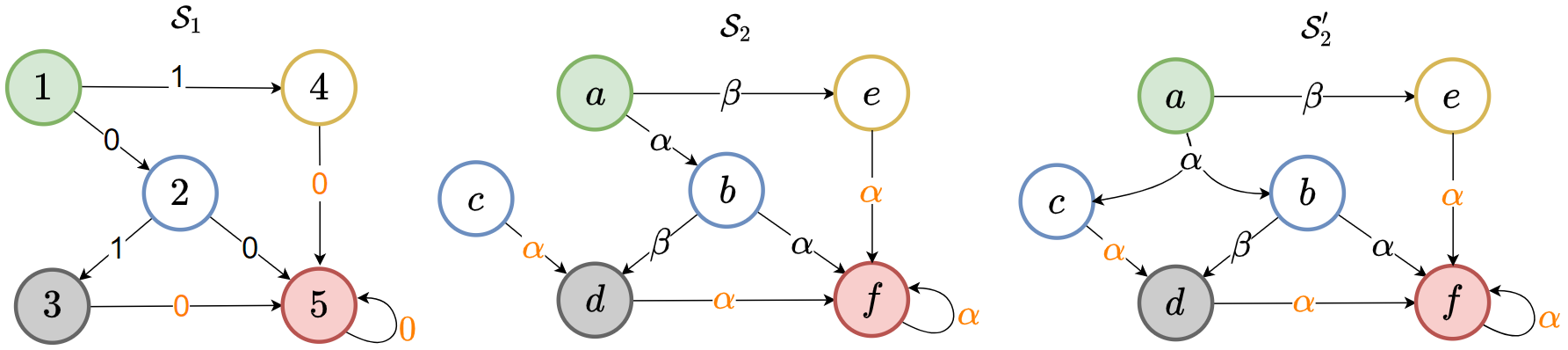}
    \caption{
    Three transition systems $\Sys_1= (\set{X}_1, \set{U}_1 , F_1)$, $\Sys_2= (\set{X}_2, \set{U}_2 , F_2)$ and $\Sys_{2}'= (\set{X}_2, \set{U}_2, F_{2}')$ and a relation $R\subseteq \set{X}_1\times\set{X}_2$. Specifically, $\set{X}_1 = (1,2,3,4,5)$, $\set{U}_1 = \{0,1\}$, $\set{X}_2 = (a,b,c,d,e,f)$, $\set{U}_2 = \{\alpha, \beta, \gamma\}$, the transition maps $F_1$, $F_2$ and $F_{2}'$ and the available inputs maps $\set{U}_1(.)$, $\set{U}_{2}(.)$ and $\set{U}_{2}'(.)$ are clear from the illustration, and $R = \{(1,a), (2,b),(2,c),(3,d),(4,e),(5,f)\}$. The colors of the node indicate the related states. Inputs labeled orange simply indicate that their values are not relevant to the current discussion. We consider the concrete specification $\Sigma_1^{\text{Reach}} =[\{1\}, \{5\}, \{3\}]$ and the associated abstract specification $\Sigma_2^{\text{Reach}} = \Sigma_{2}^{\text{Reach}'} =[\{a\}, \{f\}, \{d\}]$, where the initial, target and obstacle states are respectively in green, red and black. 
    }
    \label{fig:WFRR-motivation-1}
\end{figure}

\begin{figure}
    \centering
    \includegraphics[width=0.45\textwidth]{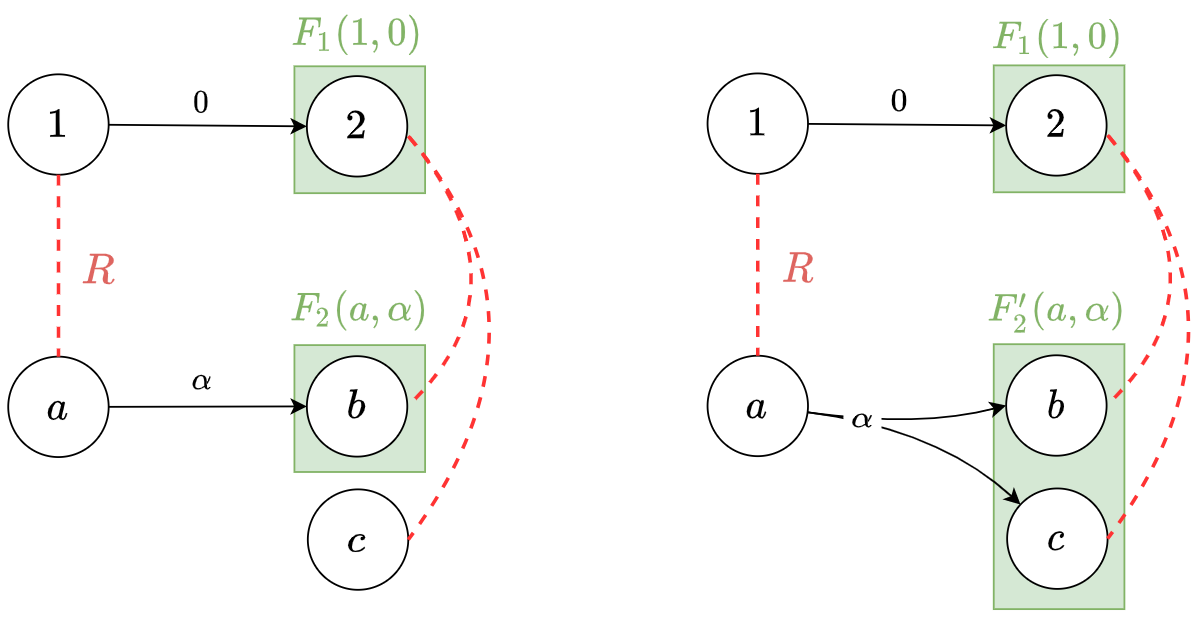}
    \caption{Local transition of two related states $(1,a)\in R$ and relation $R$ defined in~\Cref{fig:WFRR-motivation-1}. Left: Transition for $\Sys_2$. Right: Transitions for $\Sys_2'$.}
    \label{fig:WFRR-motivation-2}
\end{figure}

% Systems
We consider the two deterministic systems $\Sys_1$ and $\Sys_2$ given in~\Cref{fig:WFRR-motivation-1}.
One can verify from
~\Cref{def:R1:ASR} that the relation $R$ is an \relOne, i.e., $\Sys_1 \preceq_{R}^{\relOneAbr} \Sys_2$. 
Note that this relation corresponds to a cover-based abstraction, since $R(2) = \{b,c\}$ is not a singleton.
% Control problem
We consider the concrete specifications $\Sigma_1$ consisting in controlling $\Sys_1$ from the initial state $1$ to the target state $5$ while avoiding the obstacle state $3$.
The associated abstract specifications $\Sigma_2$ consist in controlling $\Sys_2$
 from the initial state $a$ to the target state $f$ while avoiding the obstacle state~$d$. The controller $\Cont_2$ satisfying
\begin{equation}\label{eq:controller-1}
\Cont_2(a) = \{\alpha\}, \ \Cont_2(b)= \{\alpha\}
\end{equation}
solves the abstract problem.
The maximal interface associated to $R$ is given by
$$\modif{I_R^{\relOneAbr}}(1, a, \alpha) = \{0\}, \ \modif{I_R^{\relOneAbr}}(2, b,\alpha) = \{0\}, \ \modif{I_R^{\relOneAbr}}(2, c,\alpha) = \{1\}.$$
We consider the associated \controllerTwo{}, $\Cont_1 = \Cont_2 \circ_{\modif{I_R^{\relOneAbr}}} R$. We compute the relevant values:
$$\Cont_1(1) = \{0\}, \ \Cont_1(2) = \{0, 1\}.$$
We can construct the trajectory $(\seq x_1,\seq u_1)\in\set{B}(\Cont_1 \times \Sys_1)$ with $\seq x_1=(1,2,3)$ and $\seq u_1= (0,1)$ which leads to a contradiction with the dynamics of $\Cont_2\times \Sys_2$ since the \propOne{} requires the existence of an abstract trajectory $(\seq x_2,\seq u_2)$ such that $x_2(0) = a$ and $x_2(2) = d$. 
\modif{
This shows that $\Cont_1$ does not satisfy~\eqref{eq:propOneSet}.
}
%This shows that the tuple $(\Sys_1,\Sys_2,\Cont_1,\Cont_2,R)$ 
% does not satisfy the \propOne{}.  
As a result, this specific controller does not offer the formal guarantee of avoiding the obstacle and reaching the target. 

% However, the tuple $(\Sys_1,\Sys_2,\Cont_1',\Cont_2,R)$ satisfies the \propOne{} with the concrete controller $\Cont'_1$ 
\modif{However, the controller $\Cont_1'$ defined as
$$\Cont'_1(1) = \{0\}, \ \Cont'_1(2) = \{0\}$$
satisfies~\eqref{eq:propOneSet}
(note that its existence was guaranteed by~\Cref{th:ASR:sufficient}, but that there was no guarantee that it was static).}
Indeed, the trajectory $(\seq x_1',\seq u_1')\in \set{B}(\Cont'_1\times\Sys_1)$ with $\seq x_1' = (1,2,5)$ and $\seq u_1'=(0,0)$ aligns with a corresponding abstract trajectory $(\seq x_2', \seq u_2')\in \set{B}(\Cont_2\times \Sys_2)$ with $\seq x_2'=(a,b,f)$ and $\seq u_2'=(\alpha,\alpha)$, while the previously defined trajectory $(\seq x_1, \seq u_1)\notin \set{B}(\Cont'_1\times \Sys_1)$. 

On the other hand, if we adopt an abstract controller defined as follows
\begin{equation}\label{eq:controller-2}
\tilde{\Cont}_2(a) = \{\beta\}, \ \tilde{\Cont}_2(e) = \{\alpha\}
\end{equation}
which also solves the abstract problem, then 
\modif{the controller $\tilde{\Cont}_1$, where $\tilde{\Cont}_1 = \tilde{\Cont}_2\circ_{\modif{I_R^{\relOneAbr}}} R$ is the associated \controllerTwo{}, satisfies~\eqref{eq:propOneSet}.}
% the tuple $(\Sys_1,\Sys_2,\tilde{\Cont}_1,\tilde{\Cont}_2,R)$ satisfies the \propOne{} where $\tilde{\Cont}_1 = \tilde{\Cont}_2\circ_{\modif{I_R^{\relOneAbr}}} R$ is the associated \controllerTwo{}.

Nevertheless, we aim to establish conditions that guarantee that the \controllerTwo{} ensures the \propOne{} whatever the abstract controller.

 In the subsequent subsection, we will introduce a relation that is not only \emph{sufficient} to ensure this specific concretization scheme but is also  \emph{necessary} to ensure its feasibility for every abstract controller.

%%%%%%%%%%%%%%%%%%%%%%%%%%%%%%%%%%%%%%%%%%%%%%%%%%%%%%%%%%%%%%%%%%%%%%%%%%%%%%%%%%%%%%%%%%%%%%%%%%%%%%%%%%%%%%%%%%%%%%%%%%%%%%

\subsection{Definition and properties}

The crucial point in the failure of the previous example is that the $\relOneAbr$ condition only imposes that for each transition from $x_1$ to $x_1'$ in $\Sys_1$ there exists a state $x_2' \in R(x_1')$ that is a successor of $x_2$ in $\Sys_2$, but it is not required that every $x_2''\in R(x_1')$ succeeds $x_2$ (note that it was the case in the previous example as illustrated in~\Cref{fig:WFRR-motivation-2}). %
The relation presented below is introduced to circumvent the specific problem describe previously.
\begin{definition}[$\relTwoAbr$]\label{def:WFRR}
 A relation $R \subseteq \set{X}_1\times \set{X}_2$ is a \emph{\relTwo} from $\Sys_1$ to $\Sys_2$ if 
 for every $(\vect x_{1},\vect x_{2}) \in R$ for every $\vect u_2\in \set{U}_2(x_2)$ there exists $\vect u_1\in\set{U}_1(x_1)$ such that for every $\vect x_1'\in F_1(x_1,u_1)$ for every $\vect x_2'$ such that $(\vect x_1',\vect x_2')\in R$: $ \vect x_2'\in F_2(x_2,u_2)$.
\end{definition}
The local condition stated in~\Cref{def:WFRR} is illustrated for some $(x_1,x_2,u_1,u_2)\in R_e^{\relTwoAbr}$ in Figure~\ref{fig:local-condition} (right).
The fact that $R$ is a \relTwo{} from $\Sys_1$ to $\Sys_2$ will be denoted $\Sys_1 \preceq_R^{\relTwoAbr} \Sys_2$, and we write $\Sys_1 \preceq^{\relTwoAbr} \Sys_2$ if $\Sys_1 \preceq_R^{\relTwoAbr} \Sys_2$ holds for some $R$.

The \relTwo{} is a relaxed version of the \relThree{} introduced in~\cite[V.2 Definition]{reissig2016feedback}. Specifically, we relax the requirement that the concrete and abstract inputs must be identical.

\begin{theorem}\label{prop:WFRR-ASR}
Given two systems, $\Sys_1$ and $\Sys_2$, and a strict relation~$R$, the following statements hold
\begin{enumerate}
    \item[(i)]  If $\Sys_1\preceq_R^{\relTwoAbr} \Sys_2$, then  $\Sys_1\preceq_R^{\relOneAbr} \Sys_2$.
    \item[(ii)]  Additionally, if $R$ is single-valued (i.e., defines a partition), the converse is true. %If $\Sys_1\preceq_R^{\relOneAbr} \Sys_2$, then $\Sys_1\preceq_R^{\relTwoAbr} \Sys_2$.
\end{enumerate}
\end{theorem}
\begin{proof}
\hspace{1cm}
\begin{enumerate}
    \item[(i)] Rewriting the definition slightly, we have\begin{itemize}
        \item  $\relOneAbr$:
        for every $(\vect x_{1},\vect x_{2}) \in R$ and for every $\vect u_2\in \set{U}_2(x_2)$ there exists $\vect u_1\in \set{U}_1(x_1)$ such that for every $\vect x_1'\in F_1(x_1,u_1): R(x_1')\cap F_2(x_2,u_2) \neq \myemptyset$.
       \item $\relTwoAbr$: for every $(\vect x_{1},\vect x_{2}) \in R$
       and for every $\vect u_2\in \set{U}_2(x_2)$ there exists $\vect u_1\in \set{U}_1(x_1)$ such that for every $\vect x_1'\in F_1(x_1,u_1): R(x_1')\subseteq F_2(x_2,u_2)$.
    \end{itemize}
Since $R$ is strict, $R(x_1')$ is nonempty and therefore $R(x_1')\subseteq F_2(x_2,u_2)$ implies $ R(x_1')\cap F_2(x_2,u_2) \neq \myemptyset$.

    \item[(ii)]
   Given that $R$ is a strict relation and single-valued, it can be established that for any $\vect x_1 \in \set{X}_1$: $|R(\vect x_1)| = 1$. Therefore $R(x_1')\cap F_2(x_2,u_2)\neq \myemptyset$ if and only if $R(x_1')\subseteq F_2(x_2,u_2)$.
\end{enumerate}
\end{proof}
Note that the condition that the relation forms a partition is sufficient to guarantee that an $\relOneAbr$ is a $\relTwoAbr$, but it is not a necessary condition.

%%%%%%%%%%%%%%%%%%%%%%%%%%%%%%%%%%%%%%%%%%%%%%%%%%%
%%%%%%% WFRR est réflexive et transitive %%%%%%%%%%
%%%%%%%%%%%%%%%%%%%%%%%%%%%%%%%%%%%%%%%%%%%%%%%%%%%
We prove that $\preceq^{\relTwoAbr}$ is reflexive and transitive, which justifies the use of the pre-order symbol $\preceq$. 
\begin{proposition}\label{prop:WFRR:refltrans}
Let $\Sys_1, \Sys_2$ and $\Sys_3$ be transition systems, and $R,Q$ be strict relations. Then 
\begin{enumerate}
    \item $\Sys_1 \preceq_{Id_{\set{X}_1}}^{\relTwoAbr} \Sys_1$;
    \item If $\Sys_1 \preceq_R^{\relTwoAbr} \Sys_2$ and $\Sys_2 \preceq_Q^{\relTwoAbr} \Sys_3$, then $\Sys_1 \preceq_{R\circ Q}^{\relTwoAbr} \Sys_3$.
\end{enumerate}
\end{proposition}
\begin{proof}
The identity relation $Id_{\set{X}_1}$ satisfies the definition of 
$\relTwoAbr$ with $\Sys_1 = \Sys_2$, $\vect x_1=\vect x_2$ and $\vect u_1=\vect u_2$, which proves~(1).
To prove~(2), assume that $\Sys_1 \preceq_R^{\relTwoAbr} \Sys_2 \preceq_Q^{\relTwoAbr} \Sys_3$. Then $R\circ Q$ is strict since both $R$ and $Q$ are strict.
Let $(\vect x_1,\vect x_3) \in R\circ Q$. Then there exists $\vect x_2\in\set{X}_2$ such that $(\vect x_1,\vect x_2)\in R$ and $(\vect x_2,\vect x_3)\in Q$.
Since both $Q$ and $R$ are $\relTwoAbr$, then we have
\begin{align*}
    &\forall \vect u_3\in \set{U}_3(x_3) \exists \vect u_2 \in \set{U}_2(x_2): \ Q(F_2(x_2,u_2))\subseteq F_3(x_3,u_3);\\
    &\forall \vect u_2\in \set{U}_2(x_2) \exists \vect u_1 \in \set{U}_1(x_1): \ R(F_1(x_1,u_1))\subseteq F_2(x_2,u_2);
\end{align*}
from which follows 
$$\forall \vect u_3\in \set{U}_3(x_3) \exists \vect u_1 \in \set{U}_1(x_1): \ Q(R(F_1(x_1,u_1)))\subseteq F_3(x_3,u_3)$$
and so $\Sys_1 \preceq_{R\circ Q}^{\relTwoAbr} \Sys_3$.
\end{proof}

%%%%%%%%%%%%%%%%%%%%%%%%%%%%%%%%%%%%%%%%%%%%%%%%%%%
%%%%%%%%%%%%%%% composition def %%%%%%%%%%%%%%%%%%%
%%%%%%%%%%%%%%%%%%%%%%%%%%%%%%%%%%%%%%%%%%%%%%%%%%%

%%%%%%%%%%%%%%%%%%%%%%%%%%%%%%%%%%%%%%%%%%%%%%%%%%%
%%%%%%%%%%%%%%%%%%%% property %%%%%%%%%%%%%%%%%%%%%
%%%%%%%%%%%%%%%%%%%%%%%%%%%%%%%%%%%%%%%%%%%%%%%%%%%

\begin{property}[\PropTwo]\label{prop:WFRR}
\modif{Given two systems $\Sys_1$ and $\Sys_2$, and a strict relation $R\subseteq \set{X}_1\times \set{X}_2$ of type $\T$,  
% we say that the tuple $(\Sys_1, \Sys_2, \Cont_2, R)$ satisfies the \emph{\propTwo{}} if 
we say that the tuple $(\Sys_1, \Sys_2, R)$ satisfies the \emph{\propTwo{}} if if every controller $\Cont_2$
\begin{equation}\label{eq:propTwoSet}
    R(\set{B}_x(\Cont_1\times \Sys_1))\subseteq \set{B}_x(\Cont_2\times \Sys_2)
\end{equation}
with $\Cont_1 = \Cont_2\circ_{\modif{I_R^{\T}}} R$.}
\end{property}
This property guarantees that the controller $\Cont_1 = \Cont_2\circ_{\modif{I_R^{\T}}} R$ can only generate abstract trajectories that belongs to the behavior of the system $\Cont_2\times \Sys_2$. 
To compare and contrast with~\Cref{prop:reproductibility}, we require two different things.
The first is that every quantization of a concrete trajectory is a valid abstract trajectory, as opposed to the existence of one related abstract trajectory. This is preferable if the quantizer is assumed adversarial, or alternatively, if you are free to choose the quantization that suits you best, e.g., the fastest to compute. 
The second thing that \Cref{prop:WFRR} requires is that \eqref{eq:propTwoSet} holds for a very specific controller, which gives us a straightforward implementation.

%%%%%%%%%%%%%%%%%%%%%%%%%%%%%%%%%%%%%%%%%%%%%%%%%%%
%%%%%%%%%%%%%%%%% WFRR => property %%%%%%%%%%%%%%%%
%%%%%%%%%%%%%%%%%%%%%%%%%%%%%%%%%%%%%%%%%%%%%%%%%%%

\begin{theorem}\label{th:WFRR:sufficient}
    \modif{Given two transition systems $\Sys_1$ and $\Sys_2$, and a strict relation $R$, if $\Sys_1\preceq_R^{\relTwoAbr} \Sys_2$ then   
    % for all controller $\Cont_2$, the tuple $(\Sys_1, \Sys_2, \Cont_2, R)$ satisfies the \emph{\propTwo{}} (\Cref{prop:WFRR}).
    the tuple $(\Sys_1, \Sys_2, R)$ satisfies the \emph{\propTwo{}} (\Cref{prop:WFRR}).}
\end{theorem}
\begin{proof}

Consider a trajectory $(\seq x_1, \seq u_1)\in \set{B}(\Cont_1 \times \Sys_1)$ of length $T$ with $\Cont_1 = \Cont_2\circ_{\modif{I_R^{\relTwoAbr}}} R$, i.e.,
\begin{enumerate}
    \item[(1)] $\forall k\in [0;T[:\ x_1(k+1)\in F_1(x_1(k),u_1(k))$;
    \item[(2)] $\forall k\in [0;T-1[:\ u_1(k)\in \Cont_1(x_1)$.
\end{enumerate}

Then by definition of~$\Cont_1$, for any related sequence $\seq x_2$ of length $T$ such that $\forall k\in [0;T[: \ x_2(k) \in R(x_1(k))$, there exists a sequence $\seq u_2$ of length $T-1$ such that 
\begin{enumerate}
    \item[(1')] $\forall k\in [0;T-1[: \ u_2(k)\in \Cont_2(x_2(k))$;
    \item[(2')] $\forall k\in [0;T-1[: \ u_1(k)\in \modif{I_R^{\relTwoAbr}}(\vect x_1(k),\vect x_2(k),\vect u_2(k))$.
\end{enumerate}
because $\Cont_2$ and $\modif{I_R^{\relTwoAbr}}$ are non-empty.
% because $\forall x_2\in\set{X}_2: \Cont_2(x_2)\ne \myemptyset$ and $\forall (x_1,x_2)\in R \forall u_2\ \set{U}_2(x_2):\ I_R(x_1,x_2,u_2)\ne \myemptyset$.

Our goal is to establish that $(\seq x_2,\seq u_2)\in \set{B}(\Cont_2\times \Sys_2)$, which translates to
\begin{enumerate}
    \item[(1'')] $\forall k\in [0;T-1[:\ u_2(k)\in \set{U}_2(x_2(k))$;
    \item[(2'')] $\forall k\in [0;T-1[:\ x_2(k+1)\in F_2(x_2(k),u_2(k))$.
\end{enumerate}
To prove (1''), we can directly use condition (1').
To prove (2''), taking into account condition (2') and $\Sys_1\preceq_{R}^{\relTwoAbr} \Sys_2$, we can observe that
\begin{equation}\label{eq:proof:WFRR}
    \forall x_1'\in F_1(x_1(k),u_1(k)): \ (x_1',x_2')\in R \Rightarrow x_2'\in F_2(x_2(k),u_2(k)).
\end{equation}
 Given the preceding conditions (1'), (2'), and the established implication \eqref{eq:proof:WFRR}, it follows that $x_2(k+1)\in F_2(x_2(k), u_2(k))$.
\end{proof}

%%%%%%%%%%%%%%%%%%%%%%%%%%%%%%%%%%%%%%%%%%%%%%%%%%%
%%%%%%%%%%%%%%%%% property => WFRR %%%%%%%%%%%%%%%%
%%%%%%%%%%%%%%%%%%%%%%%%%%%%%%%%%%%%%%%%%%%%%%%%%%%
Furthermore, the existence of a \relTwo{} between the concrete system and the abstraction is not only sufficient to guarantee \propTwo{}, it is also necessary when we want it to be applicable whatever the particular specification we seek to impose on the concrete system, as established by the following theorem.

\begin{theorem}\label{th:WFRR:necessary}
    \modif{Given two systems $\Sys_1$ and $\Sys_2$, and a strict relation $R$ such that $\Sys_1\preceq_R^{\relOneAbr } \Sys_2$, 
    if $(\Sys_1,\Sys_2, R)$ satisfies the \propTwo{} (\Cref{prop:WFRR}), then $\Sys_1\preceq_R^{\relTwoAbr} \Sys_2$.
    % if for every controller $\Cont_2$ the tuple $(\Sys_1,\Sys_2,\Cont_2, R)$ satisfies the \propTwo{} (\Cref{prop:WFRR}), then $\Sys_1\preceq_R^{\relTwoAbr} \Sys_2$.
    }
\end{theorem}
\begin{proof}
We proceed by contraposition. Let's assume that $R$ does not satisfy the condition of~\Cref{def:WFRR}. This implies the existence of $(\vect x_1, \vect x_2) \in R$ and $\vect u_2 \in 
\set{U}_2(x_2)$ such that for every $\vect u_1 \in \set{U}_1(x_1)$, there exist $\vect x_1' \in F_1(x_1, u_1)$ and $\vect x_2' \in R(\vect x_1')$ such that $\vect x_2' \notin F_2(x_2, u_2)$.
We consider an abstract controller $\Cont_2$ such that $\Cont_2(x_2) = \{u_2\}$.
Let $u_1\in \modif{I_R^{\relTwoAbr}}(x_1,x_2,u_2)\subseteq \set{U}_1(x_1)$.
Let's define the sequences $\seq{\tilde{x}}_1 = (\vect x_1, \vect x_1')$, $\seq{\tilde{u}}_1 = (\vect u_1)$, $\seq{\tilde{x}}_2 = (\vect x_2, \vect x_2')$, and $\seq{\tilde{u}}_2 = (\vect u_2)$. 
We observe that $(\seq{\tilde{x}}_1, \seq{\tilde{u}}_1)\in \set{B}(\Cont_1\times \Sys_1)$ with $\Cont_1 =  \Cont_2\circ_{\modif{I_R^{\relTwoAbr}}} R$ while $(\seq{\tilde{x}}_2, \seq{\tilde{u}}_2) \notin \set{B}(\Cont_2 \times \Sys_2)$ since $\tilde{x}_2(1) \notin F_2(x_2, u_2)$.
This means that the tuple \modif{$(\Sys_1, \Sys_2, R)$} does not satisfy the \propTwo{}. This concludes the proof.
\end{proof}

From any given $\relOneAbr$ abstraction, it is possible to construct a $\relTwoAbr$ abstraction, albeit with an increase in non-determinism. This augmentation is achieved by introducing additional transitions, as formally established by the subsequent proposition.

\begin{proposition}\label{prop:ASR-to-WFRR}
    Consider two systems $\Sys_1=(\set{X}_1, \set{U}_1, F_1)$ and $\Sys_2=(\set{X}_2, \set{U}_2, F_2)$ and a relation $R$ that satisfies $\Sys_1\preceq_{R}^{\relOneAbr} \Sys_2$. Then the system $\Sys_2' = (\set{X}_2, \set{U}_2, F_2')$ such that $\forall x_2\in\set{X}_2$ and $\forall u_2\in\set{U}_2(x_2)$:
    \begin{enumerate}
        \item[(i)] $\set{U}_2'(x_2) = \set{U}_2(x_2)$;
        \item[(ii)] $F_2'(x_2, u_2) =  F_2(x_2,u_2) \cup \left(\bigcup_{(x_1,x_2,u_1,u_2)\in R_e^{\relOneAbr}} \ R(F_1(x_1, u_1))\right)$;
    \end{enumerate}
    satisfies $\Sys_2\preceq_{Id_{\set{X}_2}}^{\relTwoAbr} \Sys_2'$ and  $\Sys_1\preceq_{R}^{\relTwoAbr} \Sys_2'$.\\
We call it the $\relTwoAbr$-extension of $\Sys_2$.
\end{proposition}
\begin{proof} 
The condition of~\Cref{def:WFRR}, i.e.,  $\forall (x_2, x_2)\in Id_{\set{X}_2}$ and $\forall u_2'\in\set{U}_2'(x_2)$ there exists $u_2\in \set{U}_2(x_2)$ such that $F_2(x_2,u_2)\subseteq Id_{\set{X}_2}(F_2'(x_2,u_2'))$,
 holds by taking $u_2=u_2'$ (by (i)), and the inclusion is ensured by (ii). This conclude the proof that $\Sys_2\preceq_{Id_{\set{X}_2}}^{\relTwoAbr} \Sys_2'$.

 Since $\Sys_1\preceq_{R}^{\relOneAbr} \Sys_2$ and $\Sys_2\preceq_{Id_{\set{X}_2}}^{\relTwoAbr} \Sys_2'$, by the transitivity of $\relOneAbr{}$, we can easily conclude that $\Sys_1\preceq_{R\circ Id_{\set{X}_2}}^{\relOneAbr} \Sys_2'$, which is equivalent to $\Sys_1\preceq_{R}^{\relOneAbr{}} \Sys_2'$.
In addition, by~(ii), the $\relTwoAbr$ condition, i.e., $\forall (x_1,x_2)\in R$ and $\forall u_2\in \set{U}_2(x_2)$, $\exists u_1\in \set{U}_1(x_1): 
 \ R(F_1(x_1,u_1))\subseteq F_2'(x_2,u_2)$, holds.
 This completes the proof that $\Sys_1\preceq_{R}^{\relTwoAbr} \Sys_2'$. 
\end{proof} 
Note that condition (i) limits the addition of a transition to existing labeled transitions. 

We stress that the system $\Sys_2'=(\set{X}_2, \set{U}_2, F_2')$ of~\Cref{prop:ASR-to-WFRR} is not the only way\footnote{Intuitively, we are extending along all $u_1$ related by $\relOneAbr$ to $u_2$, when extending along only one would suffice, but in general, the choice of $u_1$ is not so well defined.} to define a $\relTwoAbr$-extension of $\Sys_2$. We chose this definition to show existence of a $\relTwoAbr$-extension.
%in the sense that it is the abstraction with the smallest number of transitions satisfying the following conditions  $\Sys_2\preceq_{Id_{\set{X}_2}}^{\relTwoAbr} \Sys_2'$ and  $\Sys_1\preceq_{R}^{\relTwoAbr} \Sys_2'$. 
%
%The following result formalize the fact that $\relTwoAbr$-extension is an idempotent transformation.
%\begin{corollary}
 %   The $\relTwoAbr$-extension is an idempotent transformation.
%\end{corollary}
%\begin{proof}
%    Consider  $\Sys_1\preceq_{R}^{\relOneAbr} \Sys_2$, the the $\relTwoAbr$-extension is $\Sys_1\preceq_{R}^{\relTwoAbr} \Sys_2'$. We can view $R$ as an $\relOneAbr$ and the compute its $\relTwoAbr$-extension. this gives us $\Sys_1\preceq_{R}^{\relTwoAbr} \Sys_2''$. By definition we have \begin{enumerate}
 %       \item[(i)] $\set{U}_2''(x_2)=\set{U}_2'(x_2) = \set{U}_2(x_2)$;
  %      \item[(ii)] $F_2''(x_2, u_2) =  F_2'(x_2,u_2) \cup \left(\bigcup_{(x_1,x_2,u_1,u_2)\in R^e} \ R(F_1(x_1, u_1))\right)=F_2(x_2,u_2) \cup \left(\bigcup_{(x_1,x_2,u_1,u_2)\in R^e} \ R(F_1(x_1, u_1))\right)=F_2'(x_2, u_2) $;
   % \end{enumerate}
    %Meaning that $\Sys_2'=\Sys_2''$.
%\end{proof}
%%%%%%%%%%%%%%%%%%%%%%%%%%%%%%%%%%%%%%%%%%%%%%%%%%%
%%%%%%%%%%%%%%%%%% concretization %%%%%%%%%%%%%%%%%
%%%%%%%%%%%%%%%%%%%%%%%%%%%%%%%%%%%%%%%%%%%%%%%%%%%
\subsection{Concretization procedure}
\begin{figure}
    \centering
    \includegraphics[width=0.42\textwidth]{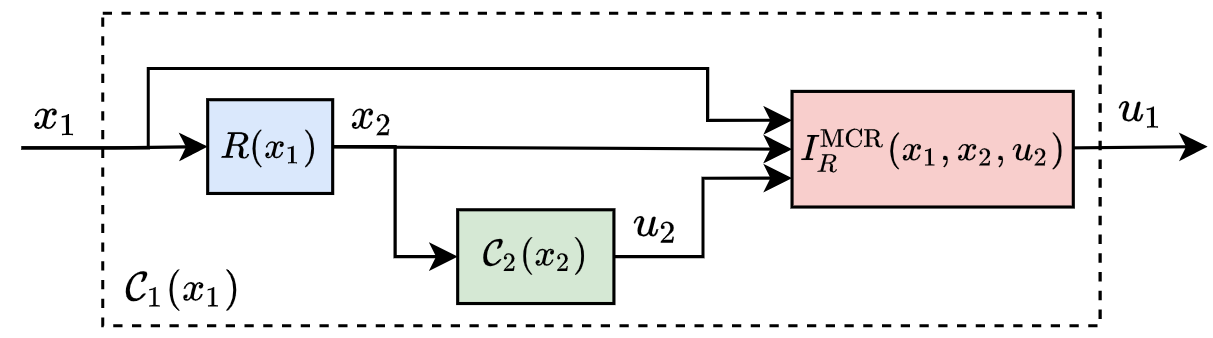}
    \caption{\ControllerTwo{} architecture $\Cont_1 = \Cont_2\circ_{\modif{I_R^{\relTwoAbr}}} R$ such that $(\Sys_1, \Sys_2, \Cont_1, \Cont_2, R)$ satisfies the \propOne{} if $\Sys_1\preceq_R^{\relTwoAbr} \Sys_2$. 
    }
    \label{fig:controller-architecture-WFRR}
\end{figure}

As already mention with the motivating example of this section, the \relTwo{} enables simpler concrete control architecture illustrated in~\Cref{fig:controller-architecture-WFRR}. 
The advantage of this control architecture is that the concrete controller only depends on the current concrete state and does not need to keep track of where the abstract system is supposed to be. In practice the abstraction is not needed once the abstract controller has been designed.
Note, if the abstract controller $\Cont_2$ is static, then the concretized controller will be static as well. 

Nevertheless, the interface $\modif{I_R^{\relTwoAbr}}$ implicitly embeds the extended relation $\modif{R_e^{\relTwoAbr}}$ into the concrete controller.
A very convenient case is when the interface is a function for which we have an explicit characterization, as we can simply compute it as needed, removing the implementation cost of storing the extended relation.
For example, consider the context of a piecewise affine controller for the concrete system, i.e., where given $x_2\in\set{X}_2$ and $u_2\in \Cont_2(x_2)$, we have
$\forall x_1\in R^{-1}(x_2): \ \modif{I_R^{\relTwoAbr}}(\vect x_1,\vect x_2,\vect u_2) = \kappa(\vect x_1)= \m K \vect x_1 + \vect l$. 
The abstract input $u_2\in \Cont_2(\vect x_2)$ can be interpreted as the local controller $\kappa$ for as long as we are in the cell $R^{-1}(x_2)\subseteq \set{X}_1$.

%%%%%%%%%%%%%%%%%%%%%%%%%%%%%%%%%%%%%%%%%%%%%%%%%%%
%%%%%%%% Go back to the motivating example %%%%%%%%
%%%%%%%%%%%%%%%%%%%%%%%%%%%%%%%%%%%%%%%%%%%%%%%%%%%

Let's now return to the motivating example of this section given in~\Cref{fig:WFRR-motivation-1}.
First note that the relation $R$ is not a $\relTwoAbr$ for systems $\Sys_1$ and $\Sys_2$ ($\Sys_1 \npreceq_{R}^{\relTwoAbr} \Sys_2$) due to the following observations:
$(1,a)\in R$, $\alpha\in \set{U}_2(a)$, $0\in \set{U}_1(1)$, $2\in F_1(1,0)$, $c\in R(2)$ and $c\notin F_2(a, \alpha)$ as 
illustrated in~\Cref{fig:WFRR-motivation-2}. 
%$(1,a,0,\alpha)\in R_e^{\relTwoAbr}$
%
We will now turn our attention to system $\Sys_2'$ introduced in~\Cref{fig:WFRR-motivation-1}. We can prove that $\Sys_1\preceq_{R}^{\relTwoAbr} \Sys_2'$ and that, consequently, \Cref{th:WFRR:sufficient} guarantees that any \controllerTwo{} will satisfy the \propOne{}. 
However, the gain in this property comes at the cost of an increase in non-determinism in the abstraction (note that $\Sys_2'$ is the~$\relTwoAbr$-extension of $\Sys_2$): whereas $\Sys_2$ was deterministic, $\Sys_2'$ is now non-deterministic. As a result, the abstract problem is harder to solve since it is no longer a simple path-finding problem on a directed graph, but a reachability problem on a weighted directed \emph{forward hypergraph}.
That is, each transition corresponds to
a \emph{forward hyperarc} which is a hyperarc with one tail and multiple heads, see \cite{gallo1993directed} for an introduction to hypergraphs.
In addition, whereas the abstract problem $(\Sys_2,\Sigma_2)$ had two solutions, the controllers $\Cont_2$ given by~\eqref{eq:controller-1} and $\Cont_2'$ given by~\eqref{eq:controller-2}, the abstract problem $(\Sys_2',\Sigma_{2}')$ has only one solution, the controller $\Cont_2'$. So, not only may the abstract problem $(\Sys_2',\Sigma_{2}')$ be more difficult to solve, it may also admit less or no solution at all. 

\subsection{Comparison between \texorpdfstring{$\relTwoAbr$}{} and \texorpdfstring{$\relThreeAbr$}{}}

As previously mentioned, the $\relTwoAbr$ extends to the \relThree~\cite[V.2 Definition]{reissig2016feedback}, whose definition we recall.

\begin{definition}[$\relThreeAbr$]\label{def:R:FRR}
A relation $R \subseteq \set{X}_1\times \set{X}_2$ is a \emph{\relThree} from $\Sys_1$ to $\Sys_2$ if 
 for every $(\vect x_{1},\vect x_{2}) \in~R$:
\begin{enumerate}
    \item $\set{U}_2(x_2)\subseteq \set{U}_1(x_1)$;
    \item for every $\vect u \in \set{U}_2(x_2)$ and for every $\vect x_1'\in F_1(x_1, u)$ for every $\vect x_2'$ such that $\vect x_2'\in R(\vect x_1')$: $ \vect x_2'\in F_2(x_2, u)$.
\end{enumerate}
\end{definition}

The fact that $R$ is a \relThree{} from $\Sys_1$ to $\Sys_2$ will be denoted $\Sys_1 \preceq_R^{\relThreeAbr} \Sys_2$, and we write $\Sys_1 \preceq^{\relThreeAbr} \Sys_2$ if $\Sys_1 \preceq_R^{\relThreeAbr} \Sys_2$ holds for some $R$.

\begin{proposition}
Given two systems $\Sys_1$ and $\Sys_2$, if $\Sys_1\preceq_R^{\relThreeAbr} \Sys_2$, then $\Sys_1\preceq_R^{\relTwoAbr}~\Sys_2$.
\end{proposition}
\begin{proof}
It's the same definition as $\relTwoAbr$, with the additional constraint that the abstract and concrete input in the extended relation are the same, that is to say $\modif{I_R^{\relTwoAbr}}(x_1,x_2,u_2)=\{u_2\}\subseteq \set{U}_1(x_1)$.
\end{proof}

%%%%%%%%%%%%%%%%%%%%%%%%%%%%%%%%%%%%%%%%%%%%%%%%%%%
%%%%%%%%%%%%%%%%%%%% refinement %%%%%%%%%%%%%%%%%%%
%%%%%%%%%%%%%%%%%%%%%%%%%%%%%%%%%%%%%%%%%%%%%%%%%%%
\begin{figure}
    \centering
    \includegraphics[width=0.42\textwidth]{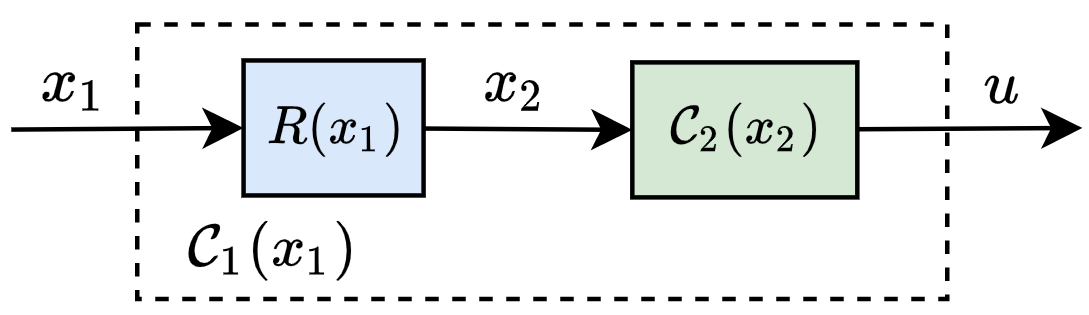}
    \caption{\ControllerTwo{} architecture $\Cont_1 = \Cont_2\circ R$ for $\Sys_1\preceq_R^{\relThreeAbr} \Sys_2$.
    }
    \label{fig:controller-architecture-FRR}
\end{figure}

In this case, the controller architecture in~\Cref{fig:controller-architecture-WFRR} simplifies with $\modif{I_R^{\relTwoAbr}}(x_1,x_2,u_2) = \{u_2\}$ in the architecture given in~\Cref{fig:controller-architecture-FRR}.
The concrete controller can be rewritten as $\Cont_1=\Cont_2\circ_{\modif{I_R^{\relTwoAbr}}} R = \Cont_2\circ R$, i.e., the functional composition of $\Cont_2$ and $R$ viewed as set-valued maps. 
This justifies the notation $\circ_{\modif{I_R^{\relTwoAbr}}}$.

This concrete controller only requires the abstract (or symbolic) state, it does not need the concrete state $x_1$; we say that the concretization is carried out using only symbolic information. 
Consequently, $\relThreeAbr$ restricts its class of concrete controllers exclusively to \emph{piecewise constant controllers}.
Therefore, the \relThree{} is well suited to the context where the exact state is not known and only quantified (or symbolic) state information is available.
On the other hand, when information on the concrete state is available, this is a major restriction, as the following example illustrates.

%%%%%%%%%%%%%%%%%%%%%%%%%%%%%%%%%%%%%%%%%%%%%%%%%%%
%%%%%%% motivating example of WFRR vs FRR %%%%%%%%%
%%%%%%%%%%%%%%%%%%%%%%%%%%%%%%%%%%%%%%%%%%%%%%%%%%%
To motivate the use of a $\relTwoAbr$ versus a $\relThreeAbr$, we provide the following example where, for a given relation $R$ (i.e., a given discretization), we can solve the concrete problem following the three-step abstraction-based procedure described in~\Cref{Sec:Intro} with a $\relTwoAbr$ whereas this is impossible with a $\relThreeAbr$, i.e., using only symbolic information for the concretization.
\begin{figure}
    \centering
    \includegraphics[width=0.46\textwidth]{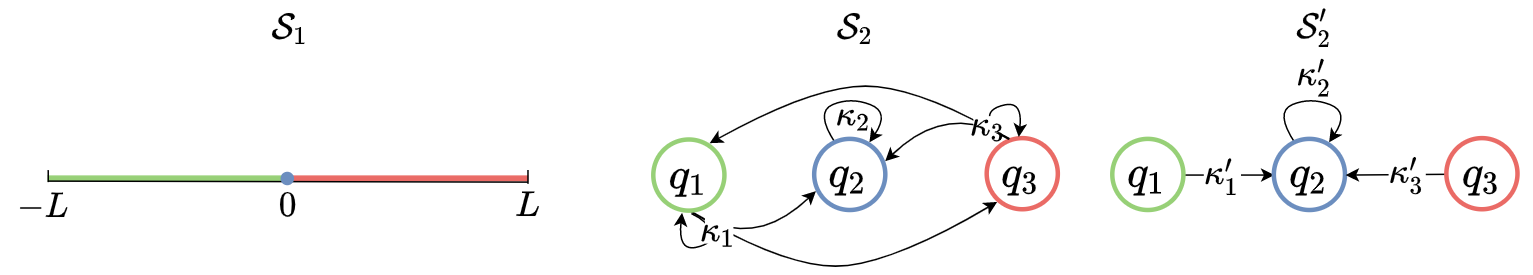}
    \caption{
    Three transition systems $\Sys_1= (\set{X}_1, \set{U}_1 , F_1)$, $\Sys_2= (\set{X}_2, \set{U}_2 , F_2)$, $\Sys_{2}'= (\set{X}_2, \set{U}_2', F_{2}')$ and a relation $R\subseteq \set{X}_1\times\set{X}_2$. Specifically, $\set{X}_1 = [-L, L]\subseteq \R$, $\set{U}_1 = \R$, $\set{X}_2 = \{q_1, q_2, q_3\}$, $\set{U}_2 = \{\kappa_1, \kappa_2, \kappa_3\}$, $\set{U}_2' = \{\kappa_1', \kappa_2', \kappa_3'\}$. The transition maps $F_2$ and $F_{2}'$ and the available inputs $\set{U}_2(.)$ and $\set{U}_2'(.)$ are clear from the illustration. Given the function $f(x,u)=x+u$, the available inputs $\set{U}_1(x_1) = \{u\in \set{U}_1 \mid f(x_1, u)\in \set{X}_1\}$ and the transition map $F_1(x_1, u) = f(x_1, u)$. Given the sets $\set{X}_{q_1}=[-L,0)$, $\set{X}_{q_2}=\{0\}$ and $\set{X}_{q_3}=(0,L]$, we define the relation $R\subseteq \set{X}_1\times \set{X}_2$ such that 
    $(x_1,q)\in R \Leftrightarrow x_1\in \set{X}_{q}$.
    The colors indicate the related states. We consider the concrete specification $\Sigma_1^{\text{Reach}} =[\set{X}_1, \{0\}, \myemptyset]$ and the associated abstract specifications $\Sigma_2^{\text{Reach}} = \Sigma_{2}^{\text{Reach}'}=[\set{X}_2, \{q_2\}, \myemptyset]$. 
    }
    \label{fig:WFRR-motivation-WFRR-vs-FRR}
\end{figure}

We consider the system $\Sys_1$ whose the dynamic consists in moving under translation on a segment of the real line, and with objective to reach $0$. 
The full definition of $\Sys_1$ is in~\Cref{fig:WFRR-motivation-WFRR-vs-FRR}.
 Any abstraction $S_2$ such that $\Sys_1\preceq^{\relThreeAbr}_R \Sys_2$ is constrained to have non-determinism for the cells $q_1$ and $q_3$, as exemplified by the system $\Sys_2$ given in~\Cref{fig:WFRR-motivation-WFRR-vs-FRR}.
 The problem here is that by using piecewise constant controllers for the concrete system, there will always be a portion around $0$ that overshoot its target. 
Indeed, we have $\kappa_1\in [0, L]$, $\kappa_2 = 0$ and $\kappa_3\in [-L,0]$, and therefore, for any admissible choice of $\kappa_1,\kappa_2,\kappa_3$, the resulting abstraction $\Sys_2$ is non-deterministic. 
 As a result, the abstract problem $(\Sys_2, \Sigma_2)$ has no solution. 
 However, we can build an abstraction $\Sys_2'$ such that $\Sys_1\preceq_R^{\relTwoAbr}\Sys_2'$ where $\kappa_1'(x_1) = \kappa_3'(x_1) = -x_1$ and $\kappa_2'(x_1) = 0$ are affine controllers. 
 Note that we have the same partition of state-space, but because we can have abstract inputs that are local state-dependent controllers instead of piecewise constant real inputs, we can solve the abstract problem $(\Sys_2',\Sigma_2')$.
For the abstract system $\Sys_2'$, the inputs $\kappa_1', \kappa_2'$ and $\kappa_3'$ can be interpreted as \emph{move to the right cell}, \emph{do not move} and \emph{move to the left cell} respectively.

By using all the inputs at our disposal and designing local state-dependent controllers, we can remove the non-determinism imposed by the discretization of the concrete system. This is crucial because abstraction-based control guarantees that the concrete controller successfully solves a control problem for the original system, provided that the abstract controller solves the associated abstract control problem, and it is worth noting that the level of non-determinism within the abstraction directly affects the feasibility of the abstract control problem. 
This example illustrates that the use of concrete state information can be beneficial to construct a \emph{practical} abstraction satisfying the \propTwo{}. 

\modif{
Finally, let us compare our work with another relation introduced in~\cite[Definition 6]{borri2018design} as the~\emph{strong alternating $\epsilon$-approximate simulation relation} (denoted $S$-$\relOneAbr$) which is characterized by symbolic concretization property, i.e., the existence of a concrete controller using only abstract state information. In fact, $S$-$\relOneAbr$ can be derived from $\relThreeAbr$ by relaxing condition (2) in~\Cref{def:R:FRR} in that of $\relOneAbr$, while $\relTwoAbr$ is obtained by relaxing condition (1) in~\Cref{def:R:FRR}.
Therefore, in the context of abstraction with overlapping cells, $S$-$\relOneAbr$ has the same drawback as $\relOneAbr$, in that it does not allow the use of the controller architecture described in~\Cref{fig:controller-architecture-FRR}, but requires the controller architecture given in~\Cref{fig:controller-architecture-ASR}.
}

% % SECTION 5 : Conclusion ---------------------------------
\section{Conclusion}
\label{Sec:Conclusion}

% 1) Decrire l'avantage de cette concrétisation
We have introduced~\relTwo{}, which provides a framework that guarantees a simple control architecture, requiring only information about the current state. 
In addition, this concretization procedure is independent of the type of dynamical systems and specifications under consideration.
In particular, if the abstract controller is static, the concrete controller will also be static. 
% 2) Dire que 
We have additionally provided a precise characterization of this relation with a necessary and sufficient condition on the concretization architecture.

We have demonstrated that any \relOne{} can be extended to a \relTwo{} at the price of introducing additional non-determinism. Furthermore, we prove that $\relOneAbr$ and $\relTwoAbr$ coincide in the particular case of a deterministic quantizer, and thus that $\relOneAbr$ benefits from the \propTwo{} in this specific case.

% 3) Dire que notre relation est un cadre bien adapté pour optimiser pour la forme des cells et utiliser des controlleurs locaux state-dependants, aus sens que dependant de l'informaton concrete. Ce qui est bien de construire une asbtraction pas trop non-determinsitiste. 
In addition, we showed that this framework allows the use of overlapping cells (referred to as cover-based abstraction) and piecewise state-dependent controllers, enabling the design of low-level controllers within cells, in combination with high-level abstraction-based controllers.
This opens up new possibilities when co-creating the abstraction and the controller, as is done in so-called \emph{lazy abstractions}.

% In future works we would like to extend the comparison to the state-feedback relation given in~\cite[Definition~1]{egidio2022state}, and see if we can characterize each of these relation by a property in a manner similar to Theorems~\ref{th:WFRR:necessary} and \ref{th:WFRR:sufficient}.

\bibliographystyle{IEEEtran}
\bibliography{main}

\begin{thebibliography}{10}
\providecommand{\url}[1]{#1}
\csname url@samestyle\endcsname
\providecommand{\newblock}{\relax}
\providecommand{\bibinfo}[2]{#2}
\providecommand{\BIBentrySTDinterwordspacing}{\spaceskip=0pt\relax}
\providecommand{\BIBentryALTinterwordstretchfactor}{4}
\providecommand{\BIBentryALTinterwordspacing}{\spaceskip=\fontdimen2\font plus
\BIBentryALTinterwordstretchfactor\fontdimen3\font minus
  \fontdimen4\font\relax}
\providecommand{\BIBforeignlanguage}[2]{{%
\expandafter\ifx\csname l@#1\endcsname\relax
\typeout{** WARNING: IEEEtran.bst: No hyphenation pattern has been}%
\typeout{** loaded for the language `#1'. Using the pattern for}%
\typeout{** the default language instead.}%
\else
\language=\csname l@#1\endcsname
\fi
#2}}
\providecommand{\BIBdecl}{\relax}
\BIBdecl

\bibitem{reissig2016feedback}
G.~Reissig, A.~Weber, and M.~Rungger, ``Feedback refinement relations for the
  synthesis of symbolic controllers,'' \emph{IEEE Transactions on Automatic
  Control}, vol.~62, no.~4, pp. 1781--1796, 2016.

\bibitem{belta2017formal}
C.~Belta, B.~Yordanov, and E.~A. Gol, \emph{Formal methods for discrete-time
  dynamical systems}.\hskip 1em plus 0.5em minus 0.4em\relax Springer, 2017,
  vol.~89.

\bibitem{kupferman2001model}
O.~Kupferman and M.~Y. Vardi, ``Model checking of safety properties,''
  \emph{Formal methods in system design}, vol.~19, pp. 291--314, 2001.

\bibitem{alur1998alternating}
R.~Alur, T.~A. Henzinger, O.~Kupferman, and M.~Y. Vardi, ``Alternating
  refinement relations,'' in \emph{CONCUR'98 Concurrency Theory: 9th
  International Conference Nice, France, September 8--11, 1998 Proceedings
  9}.\hskip 1em plus 0.5em minus 0.4em\relax Springer, 1998, pp. 163--178.

\bibitem{tabuada2009verification}
P.~Tabuada, \emph{Verification and control of hybrid systems: a symbolic
  approach}.\hskip 1em plus 0.5em minus 0.4em\relax Springer Science \&
  Business Media, 2009.

\bibitem{rungger2016scots}
M.~Rungger and M.~Zamani, ``Scots: A tool for the synthesis of symbolic
  controllers,'' in \emph{Proceedings of the 19th international conference on
  hybrid systems: Computation and control}, 2016, pp. 99--104.

\bibitem{borri2018design}
A.~Borri, G.~Pola, and M.~D. Di~Benedetto, ``Design of symbolic controllers for
  networked control systems,'' \emph{IEEE Transactions on Automatic Control},
  vol.~64, no.~3, pp. 1034--1046, 2018.

\bibitem{egidio2022state}
L.~N. Egidio, T.~A. Lima, and R.~M. Jungers, ``State-feedback abstractions for
  optimal control of piecewise-affine systems,'' in \emph{2022 IEEE 61st
  Conference on Decision and Control (CDC)}, 2022, pp. 7455--7460.

\bibitem{majumdar2020abstraction}
R.~Majumdar, N.~Ozay, and A.-K. Schmuck, ``On abstraction-based controller
  design with output feedback,'' in \emph{Proceedings of the 23rd International
  Conference on Hybrid Systems: Computation and Control}, 2020, pp. 1--11.

\bibitem{girard2012controller}
A.~Girard, ``Controller synthesis for safety and reachability via approximate
  bisimulation,'' \emph{Automatica}, vol.~48, no.~5, pp. 947--953, 2012.

\bibitem{dallal2013supervisory}
E.~Dallal, A.~Colombo, D.~Del~Vecchio, and S.~Lafortune, ``Supervisory control
  for collision avoidance in vehicular networks with imperfect measurements,''
  in \emph{52nd IEEE Conference on Decision and Control}.\hskip 1em plus 0.5em
  minus 0.4em\relax IEEE, 2013, pp. 6298--6303.

\bibitem{grune2007approximately}
L.~Grune and O.~Junge, ``Approximately optimal nonlinear stabilization with
  preservation of the lyapunov function property,'' in \emph{2007 46th IEEE
  Conference on Decision and Control}.\hskip 1em plus 0.5em minus 0.4em\relax
  IEEE, 2007, pp. 702--707.

\bibitem{hsu2018lazy}
K.~Hsu, R.~Majumdar, K.~Mallik, and A.-K. Schmuck, ``Lazy abstraction-based
  control for safety specifications,'' in \emph{2018 IEEE Conference on
  Decision and Control (CDC)}.\hskip 1em plus 0.5em minus 0.4em\relax IEEE,
  2018, pp. 4902--4907.

\bibitem{yordanov2011temporal}
B.~Yordanov, J.~Tumova, I.~Cerna, J.~Barnat, and C.~Belta, ``Temporal logic
  control of discrete-time piecewise affine systems,'' \emph{IEEE Transactions
  on Automatic Control}, vol.~57, no.~6, pp. 1491--1504, 2011.

\bibitem{girard2013low}
A.~Girard, ``Low-complexity quantized switching controllers using approximate
  bisimulation,'' \emph{Nonlinear Analysis: Hybrid Systems}, vol.~10, pp.
  34--44, 2013.

\bibitem{hsu2019lazy}
K.~Hsu, R.~Majumdar, K.~Mallik, and A.-K. Schmuck, ``Lazy abstraction-based
  controller synthesis,'' in \emph{International Symposium on Automated
  Technology for Verification and Analysis}.\hskip 1em plus 0.5em minus
  0.4em\relax Springer, 2019, pp. 23--47.

\bibitem{legat2021abstraction}
B.~Legat, R.~M. Jungers, and J.~Bouchat, ``Abstraction-based branch and bound
  approach to q-learning for hybrid optimal control,'' in \emph{Learning for
  Dynamics and Control}.\hskip 1em plus 0.5em minus 0.4em\relax PMLR, 2021, pp.
  263--274.

\bibitem{calbert2021alternating}
J.~Calbert, B.~Legat, L.~N. Egidio, and R.~Jungers, ``Alternating simulation on
  hierarchical abstractions,'' in \emph{IEEE Conference on Decision and Control
  (CDC)}, 2021, pp. 593--598.

\bibitem{gallo1993directed}
G.~Gallo, G.~Longo, S.~Pallottino, and S.~Nguyen, ``Directed hypergraphs and
  applications,'' \emph{Discrete applied mathematics}, vol.~42, no. 2-3, pp.
  177--201, 1993.

\end{thebibliography}

\end{document}